\documentclass[11pt,twoside]{article}

\usepackage{fullpage}

\usepackage{amsthm, amsmath, amssymb, hyperref, dsfont, mathtools,bm}
\usepackage{framed}
\usepackage{caption,subcaption,graphicx,rotating,multirow}
\usepackage{url}
\usepackage{color,xcolor}
\usepackage{enumitem}
\usepackage{epstopdf}

\theoremstyle{plain}
\newtheorem{theorem}{Theorem}

\newtheorem*{remark}{Remark}
\newtheorem{lemma}[theorem]{Lemma}

\newtheorem{corollary}[theorem]{Corollary}

\newtheorem{proposition}[theorem]{Proposition}
 \newenvironment{proofof}[1]{ {\em Proof of #1.}}{\hfill \rule{2mm}{2mm} 
 }
  
\newlength{\widebarargwidth}
\newlength{\widebarargheight}
\newlength{\widebarargdepth}
\DeclareRobustCommand{\widebar}[1]{%
  \settowidth{\widebarargwidth}{\ensuremath{#1}}%
  \settoheight{\widebarargheight}{\ensuremath{#1}}%
  \settodepth{\widebarargdepth}{\ensuremath{#1}}%
  \addtolength{\widebarargwidth}{-0.3\widebarargheight}%
  \addtolength{\widebarargwidth}{-0.3\widebarargdepth}%
  \makebox[0pt][l]{\hspace{0.3\widebarargheight}%
    \hspace{0.3\widebarargdepth}%
    \addtolength{\widebarargheight}{0.3ex}%
    \rule[\widebarargheight]{0.95\widebarargwidth}{0.1ex}}%
  {#1}}
  
\makeatletter
\long\def\@makecaption#1#2{
        \vskip 0.8ex
        \setbox\@tempboxa\hbox{\small {\bf #1:} #2}
        \parindent 1.5em  
        \dimen0=\hsize
        \advance\dimen0 by -3em
        \ifdim \wd\@tempboxa >\dimen0
                \hbox to \hsize{
                        \parindent 0em
                        \hfil 
                        \parbox{\dimen0}{\def\baselinestretch{0.96}\small
                                {\bf #1.} #2
                                } 
                        \hfil}
        \else \hbox to \hsize{\hfil \box\@tempboxa \hfil}
        \fi
        }
\makeatother

\newcommand{\half}{\frac{1}{2}}

\newcommand{\normal}{\ensuremath{\mathcal{N}}}

\newcommand{\bigo}{\ensuremath{\mathcal{O}}}

\renewcommand{\P}{\operatorname{\mathbb{P}}}
\newcommand{\E}{\operatorname{\mathbb{E}}}

\newcommand{\R}{\mathbb{R}}

\newcommand{\Rp}{\mathbb{R}_{+}}

\newcommand{\indi}{\mathds{1}}

\newcommand{\imnb}{\mathfrak{i}}

\newcommand{\rmd}{\mathrm{d}}

\newcommand{\abs}[1]{\left|#1\right|}
\newcommand{\vct}[1]{\bm{#1}}
\newcommand{\mtx}[1]{\bm{#1}}

\def\BC{\begin{center}}
\def\EC{\end{center}}
\def\BIT{\begin{itemize}}
\def\EIT{\end{itemize}}
\def\BET{\begin{enumerate}}
\def\EET{\end{enumerate}}
\def\BEQ{\begin{equation}}
\def\EEQ{\end{equation}}

\long\def\comment#1{}

\usepackage[authoryear]{natbib}
\newcommand{\erfc}{\mathsf{erfc}}
\newcommand{\err}{\mathsf{err}}

\newcommand{\D}{\mathcal{D}}

\newcommand{\KL}{\mathsf{KL}} 
\newcommand{\TV}{\mathsf{TV}} 
\newcommand{\RS}{\mathsf{RS}} 
\newcommand{\RSB}{\mathsf{RSB}}

\newcommand{\xtt}{\textup{\texttt{x}}}
\newcommand{\x}{{\bm x}}
\newcommand{\z}{{\bm z}}
\newcommand{\Y}{{\bm Y}}
\newcommand{\W}{{\bm W}}

\begin{document}

\title{\bf{\Large{Fundamental limits of detection in the spiked Wigner model}}}

\author{Ahmed El Alaoui\thanks{Department of EECS, UC Berkeley, CA. Email: elalaoui@berkeley.edu}
\and
Florent Krzakala\thanks{Laboratoire de Physique Statistique, CNRS, PSL Universit\'es \& Ecole Normale Sup\'erieure, Sorbonne Universit\'es et Universit\'e Pierre \& Marie Curie, Paris, France.} 
\and
Michael I. Jordan\thanks{Departments of EECS and Statistics, UC Berkeley, CA.}
}

\date{}
\maketitle

\vspace*{-.3in} 
\begin{abstract}
We study the fundamental limits of detecting the presence of an additive rank-one perturbation, or spike, to a Wigner matrix.  
When the spike comes from a prior that is i.i.d.\  across coordinates, we prove that the log-likelihood ratio of the spiked model against the non-spiked one is asymptotically normal below a certain \emph{reconstruction threshold} which is not necessarily of a ``spectral" nature, and that it is degenerate above. This establishes the maximal region of contiguity between the planted and null models. It is known that this threshold also marks a phase transition for estimating the spike: the latter task is possible above the threshold and impossible below. Therefore, both estimation and detection undergo the same transition in this random matrix model. We also provide further information about the performance of the optimal test.   
Our proofs are based on Gaussian interpolation methods and a rigorous incarnation of the cavity method, as devised by Guerra and Talagrand
in their study of the Sherrington--Kirkpatrick spin-glass model.
\end{abstract}

\section{Introduction}
\label{sxn:intro}
Spiked models, which are distributions over matrices of the form ``signal + noise," have been a
mainstay in the statistical literature since their introduction by~\cite{johnstone2001distribution}
as models for the study of high-dimensional principal component analysis.
Spectral properties of these models have been extensively studied, in particular in random matrix theory,
where they are known as \emph{deformed ensembles}~\citep{peche2014ICM}. Landmark investigations in this
area~\citep{baik2005phase,baik2006eigenvalues,peche2006largest,feral2007largest,capitaine2009largest,bai2012sample,bai2008central}
have established the existence of a \emph{spectral threshold} above which the top eigenvalue detaches
from the bulk of eigenvalues and becomes informative about the spike, and below which the top eigenvalue
bears no information. Estimation using the top eigenvector undergoes the same transition, where it is known
to ``lose track" of the spike below the spectral threshold~\citep{paul2007asymptotics, nadler2008finite,
johnstone2009consistency, benaych2011eigenvalues}. Although these spectral analyses have provided many
insights, as have analyses based on more thoroughgoing usage of spectral data and/or more advanced
optimization-based procedures~\citep[see][and references therein]{ledoit2002some,amini2008high,berthet2013optimal,dobriban2017sharp}, they do not characterize the
fundamental limits of estimating the spike, or detecting its presence, from the observation of a
sample matrix. 
Important progress on the detection problem was made by 
\cite{onatski2013asymptotic,onatski2014signal} and \cite{johnstone2015testing}, who considered the spiked covariance model for a uniformly distributed unit norm spike, and studied the asymptotics of the likelihood ratio (LR) of a spiked alternative against a spherical null. They showed asymptotic normality of the log-LR below the spectral threshold \citep[also known as the BBP threshold, after][in this setting]{baik2005phase}, while it is degenerate, i.e., exponentially small (large) under the null (alternative), above it. Their proof is intrinsically tied to the assumption of a spherical prior since it relies on the rotational symmetry of the model to express the LR exclusively in terms of the spectrum, the joint distribution of which is available in closed form.

We focus in this paper on the \emph{spiked Wigner model}, which is the following symmetric random matrix model:
\begin{equation}\label{spiked_wigner_model}
\Y = \sqrt{\frac{\lambda}{N}} \x^* \x^{*\top} + \W,
\end{equation}
where $W_{ij} = W_{ji} \sim \normal(0,1)$ and $W_{ii} \sim \normal(0,\sigma^2)$, $\sigma>0$, are independent for all $1 \le i \le j \le N$.
The \emph{spike} vector $\x^* \in \R^N$ represents the signal to be recovered, or its presence detected.

We assume that the entries $x^*_i$ of the spike are i.i.d.\ from a prior distribution $P_{\xtt}$ on $\R$ having \emph{bounded} support.
The parameter $\lambda \ge 0$ plays the role of the signal-to-noise ratio, and the scaling by $\sqrt{N}$ is such that the signal and noise components of the observed data are of comparable magnitudes.
Upon observing $\Y$, we want to test whether $\lambda>0$ or $\lambda=0$.
We moreover want to understand the performance of the likelihood ratio test, which minimizes the sum of the Type-I and Type-II errors by the Neyman-Pearson lemma.

The testing problem becomes more subtle in our setting, where the spike comes from a product prior, since it is not clear that one does not lose power by discarding the eigenvectors of $\Y$. In fact, this situation presents a richer phenomenology: while the spherical case is characterized by the behavior of the spectrum, and the spectral threshold separates the regions of convergence and degeneracy of the LR, there are priors $P_{\xtt}$ in the i.i.d.\ case for which the spectral threshold loses its information-theoretic relevance. These priors exhibit a more subtle phase transition that happens \emph{strictly before} the spike manifests its presence in the spectrum. A desire to understand this phenomenon is the main impetus for the present work.   
 
This transition was discovered by~\cite{lesieur2015phase} while studying the estimation problem in the context of sparse PCA. \cite{perry2016optimality_annals} and \cite{banks2017information} proved the possibility of both estimation and asymptotically certain---we will say ``strong"---detection below the spectral threshold for certain sparse priors. However, their techniques---which are based on careful conditioning of the second moment of the LR---are not able to determine the phase transition threshold, the explicit form of which was conjectured by Lesieur et al. 

Our contribution is to rigorously pin down this phase transition for the detection problem. We prove asymptotic normality of the log-LR below a certain \emph{reconstruction threshold} $\lambda_c$ and degeneracy above it. This allows us to show mutual contiguity of the null and the alternative below $\lambda_c$ and to derive formulas for the Type-I and Type-II errors of the LR test, as well as the KL divergence and total variation distance, between the null and alternative. Our approach reposes on seminal work by Guerra and Talagrand in their study of the Sherrington-Kirkpatrick (SK) spin-glass model. 

The paper is organized as follows: Section~\ref{sxn:lr_rs_rec} sets up the problem, Section~\ref{sxn:fluctuations_wigner} contains our main results on LR fluctuations and the limits of detection, Section~\ref{sxn:replicas_nishimori} provides background on essential concepts from spin-glass theory that are necessary for the proof, and  Sections~\ref{sxn:proof_fluctuations}, \ref{sxn:asymptotic_decoupling} and~\ref{sxn:overlap_convergence} are devoted to the detailed proofs.

\section{The LR, the $\RS$ formula and the reconstruction threshold}
\label{sxn:lr_rs_rec}
\subsection{The LR} 
We denote by $\P_{\lambda}$ the joint probability law of the observations, $\Y=\{Y_{ij} : 1\le i \le j \le N\}$, as per~\eqref{spiked_wigner_model} and we define the likelihood ratio or Radon-Nikodym derivative of $\P_{\lambda}$ with respect to $\P_0$ as
\begin{equation}\label{likelihood_ratio}
L(\cdot;\lambda) \equiv  \frac{\mathrm{d}\mathbb{P}_{\lambda}}{\mathrm{d}\mathbb{P}_{0}}.
\end{equation}
Conditioning on $\x^*$ and using the Gaussianity of $\W$ yields the formula
\begin{equation}\label{likelihood_ratio_explicit}
L(\Y;\lambda) = \int \exp\Big(\sum_{i < j}\sqrt{\frac{\lambda}{N}} Y_{ij}x_ix_j -\frac{\lambda}{2N} x_i^2x_j^2 
+ \frac{1}{\sigma^2}\sum_{i=1}^N \sqrt{\frac{\lambda}{N}} Y_{ii}x_i^2 -\frac{\lambda}{2N} x_i^4
\Big) ~\rmd P_{\xtt}^{\otimes N}(\x),
\end{equation}
for any fixed $\Y$.
Define the \emph{free energy} of the planted model $\P_{\lambda}$ as
\begin{equation}\label{free_energy}
f_N := \frac{1}{N} \E_{\P_{\lambda}}\log L(\Y;\lambda) = \frac{1}{N} D_{\KL}(\P_\lambda, \P_0),
\end{equation}
where $D_{\KL}$ is the Kullback-Leibler divergence between probability measures.     
The reconstruction threshold $\lambda_c$ is defined as the largest positive number below which the limit of $f_N$ vanishes. 
This latter limit, referred to as the \emph{replica-symmetric} ($\RS$) formula, provides a full characterization
of the limits of estimating the spike with non-trivial accuracy~\citep{barbier2016mutual,lelarge2016fundamental}. 

\subsection{The $\RS$ formula} For $r \ge 0$, consider the function
\begin{equation}\label{log_partition_scalar_gaussian_channel}
\psi(r) := \E_{x^*,z} \log \int \exp\left(\sqrt{r}zx + r xx^* - \frac{r}{2} x^2\right) \rmd P_{\xtt}(x),
\end{equation}
where $z \sim\normal(0,1)$, and $x^* \sim P_{\xtt}$.
This is the $\KL$ divergence between the distributions of the random variables $y = \sqrt{r}x^*+z$ and $z$. We define the replica-symmetric potential
\begin{equation}\label{RS_potential}
F(\lambda, q) := \psi(\lambda q) -\frac{\lambda q^2}{4},
\end{equation}
and the replica-symmetric formula
\begin{equation}\label{RS_formula}
\phi_{\RS}(\lambda) := \sup_{q \ge 0}~ F(\lambda, q).
\end{equation}
A central result in this context, which was conjectured by~\cite{lesieur2015phase},
and then proved in a sequence of papers~\citep{deshpande2016asymptotic,barbier2016mutual,krzakala2016mutual,lelarge2016fundamental,alaoui2018estimation},
is that free energy $f_N$ converges to the $\RS$ formula for all $\lambda \ge 0$:
\begin{equation}\label{convergence_free_energy}
f_N ~~ \longrightarrow ~~\phi_{\RS}(\lambda).
\end{equation}
In particular the limit is independent of $\sigma$, i.e., it is insensitive to $(Y_{ii})_{i=1}^N$.

The values of $q$ that maximize the $\RS$ potential and their properties play an important role. \cite{lelarge2016fundamental}
proved that the map $q \mapsto F(\lambda,q)$ has a unique maximizer $q^* = q^*(\lambda)$ for all $\lambda \in \D$ where $\D = \R_+ \setminus \mbox{countable set}$. Moreover, they showed that the map $ \lambda \in \D \mapsto q^*(\lambda)$ is non-decreasing, and
\begin{equation}\label{limits_q_star}
\lim_{\underset{\lambda \in \D}{\lambda \to 0}} q^*(\lambda) = \E_{P_{\xtt}}[X]^2, \qquad \text{and} \qquad \lim_{\underset{\lambda \in \D}{\lambda \to \infty}} q^*(\lambda) = \E_{P_{\xtt}}[X^2].
\end{equation}
where $X \sim P_{\xtt}$.
One can interpret the value $q^*(\lambda)$ as the best overlap an estimator $\widehat{\theta}(\Y)$ based on observing $\Y$ can have with the spike $\x^*$. Indeed, Lelarge and Miolane also showed that the squared overlap $(\frac{1}{N}\x^\top\x^*)^2$ between the spike $\x^*$ and a random draw $\x$ from the posterior $\P_\lambda(\cdot | \Y)$ concentrates about $q^*(\lambda)^2$.

\subsection{The reconstruction threshold} 

The first limit in~\eqref{limits_q_star} shows that when the prior $P_{\xtt}$ is not centered, it is always possible to have a non-zero overlap with $\x^*$ (just by guessing at random from the prior). An interesting situation then is when the prior has zero mean. Since $q^*$ is a non-decreasing function of $\lambda$, it is useful to define the critical value of $\lambda$ below which a non-zero overlap with $\x^*$ is impossible:
\begin{align}\label{lambda_critical}
\lambda_c &:= \sup \big\{ \lambda >0 ~:~ q^*(\lambda) =0\big\}\\
&~=\sup \big\{ \lambda >0 ~:~ \phi_{\RS}(\lambda) =0\big\}\nonumber.
\end{align}
The second equality follows by the a.e.\ uniqueness of the maximizer $q^*$.
We refer to $\lambda_c$ as the \emph{reconstruction threshold}. 
The next lemma establishes a natural bound on $\lambda_c$.
\begin{lemma} \label{spectral_bound}
We have
$\lambda_c \cdot \left(\E_{P_{\xtt}}[X^2]\right)^2 \le 1$.
\end{lemma}
\begin{proof}
Indeed, assume that $P_{\xtt}$ is centered, and let $\lambda > (\E[X^2])^{-2}$. Since $\psi'(0) = \half\E_{P_{\xtt}}[X]^2 = 0$ and $\psi''(0) = \half(\E_{P_{\xtt}}[X^2])^2$, we see that $\partial_q F(\lambda, 0) = 0$ and $\partial_q^2 F(\lambda,0) = \frac{\lambda}{2}(\lambda \E_{P_{\xtt}}[X^2]^2 - 1)>0$. So $q=0$ cannot be a maximizer of $F(\lambda,\cdot)$. Therefore $q^*(\lambda)>0$ and $\lambda \ge \lambda_c$.
\end{proof}
The importance of Lemma~\ref{spectral_bound} stems from the fact that the value $\left(\E_{P_{\xtt}}[X^2]\right)^{-2}$ is the spectral threshold previously discussed. Above this value, the first eigenvalue of the matrix $\Y$
detaches from the bulk~\citep{peche2006largest,capitaine2009largest,feral2007largest}. This value also marks the limit below which the first eigenvector of $\Y$ captures no information about the spike $\x^*$~\citep{benaych2011eigenvalues}.
The inequality in Lemma~\ref{spectral_bound} can be strict or turn into equality depending on the prior $P_{\xtt}$.
For instance, there is equality if the prior is Gaussian or Rademacher---so that the first eigenvector overlaps with the spike as soon as estimation becomes possible at all---and strict inequality in the case of the (sufficiently) sparse Rademacher prior $P_{\xtt} = \frac{\rho}{2}\delta_{-1/\sqrt{\rho}} + (1-\rho)\delta_{0} + \frac{\rho}{2}\delta_{+1/\sqrt{\rho}}$. More precisely, there exists a value
\[\rho^* = \inf \big\{\rho \in (0,1) ~:~ \psi'''(0)<0 \big\} \approx 0.092,\]
such that $\lambda_c = 1$ for $\rho \ge \rho^*$, and $\lambda_c < 1$ for $\rho < \rho^*$.
In the latter case, the spectral approach to estimating $\x^*$ fails for $\lambda \in (\lambda_c, 1)$, and it is believed that no polynomial time algorithm succeeds in this region~\citep{lesieur2015phase,krzakala2016mutual,banks2017information}.

\section{Fluctuations below the reconstruction threshold}
\label{sxn:fluctuations_wigner}

In this section we study the behavior of $\log L$. It can be seen by a standard concentration-of-measure argument that for all $\lambda > 0$, $\log L(\Y;\lambda)$ concentrates about its expectation with fluctuations of order $\sqrt{N}$. While this bound is likely to be of the right order above $\lambda_c$, it is very pessimistic below $\lambda_c$. Indeed, we will show that the fluctuations are of constant order with a Gaussian limiting law in this regime. This behavior of unusually small fluctuations is often referred to as ``super-concentration." We refer to \cite{chatterjee2014superconcentration} for more on this topic. 
Throughout the rest of the paper, except in Section~\ref{sxn:diagonal_kept}, we discard the diagonal terms $Y_{ii}$ from the observations: we formally take $\sigma=+\infty$ in~\eqref{likelihood_ratio_explicit}. (See the Remark below).
\begin{theorem} \label{central_limit_theorem}
Assume that the prior $P_{\xtt}$ is centered, has unit variance and bounded support. Also, let $\sigma=+\infty$.
For all $\lambda <\lambda_c$,
\[\log L(\Y;\lambda) \rightsquigarrow \normal\left(\pm\frac{1}{4} \left(-\log(1-\lambda)-\lambda\right),\frac{1}{2} \left(-\log(1-\lambda)-\lambda\right)\right),\]
 where the plus sign holds under the alternative $\Y \sim \P_{\lambda}$ and the minus sign under the null $\Y \sim \P_{0}$. The symbol $``\rightsquigarrow"$ denotes convergence in distribution as $N \to \infty$.
\end{theorem}
\begin{remark}
The assumption $\sigma=+\infty$ is only for convenience; its removal does not pose any additional technical difficulties. 
When the diagonal is kept, the limiting Gaussian is still of the form $\normal(\pm\mu,2\mu)$, but now  $\mu = \frac{1}{4} \left(-\log(1-\lambda)-\lambda\right)(1+\frac{\kappa}{\sigma^2}) + \frac{\lambda}{2\sigma^2}$, $\kappa = \E_{P_{\xtt}}[X^3]^2$. We refer to Section~\ref{sxn:diagonal_kept} for a discussion of how this adjusted formula would appear in the proof.   
\end{remark}
We point out that a result of this form was originally proved in the case of the Sherrington--Kirkpatrick (SK) model: \cite{aizenman1987some} showed that the log-partition function of this model has Gaussian fluctuations in the ``high temperature" regime (which corresponds to $\lambda$ small enough.) In fact, Theorem~\ref{central_limit_theorem}, if specialized to the Rademacher prior $P_{\xtt} = \half \delta_{+1} + \half\delta_{-1}$, reduces to their result (with $\lambda_c=1$) since the LR $L$ is equal to the partition function of the SK model in that case.

Our result has a parallel in the work of~\cite{johnstone2015testing,onatski2013asymptotic,onatski2014signal},
who focused on sperical priors and studied the likelihood ratio of the joint eigenvalue densities under the spiked covariance model,
showing its asymptotic normality below the spectral threshold.
We also note that similar fluctuation results were recently proved by~\cite{baik2016fluctuations,baik2017fluctuations} for a spherical model where one integrates over the uniform measure on the sphere in the definition of $L$. Their model, due to its integrable nature, is amenable to analysis using tools from random matrix theory.
The authors are thus able to also analyze a ``low temperature" regime (absent from our problem) where the fluctuations are no longer Gaussian  but given by the Tracy-Widom distribution. However, their techniques seem to be restricted to the spherical case.
Closer to our setting is the recent work of~\cite{banerjee2018asymptotic} \citep[see also][]{banerjee2018contiguity} who use a very precisely conditioned second-moment argument to show asymptotic normality of similar log-likelihood ratios. However, this technique (at least in its current state) is not able to achieve the optimal threshold $\lambda_c$.

\subsection{Limits of strong and weak detection}
Consider the problem of deciding whether an array of observations $\Y=\{Y_{ij} : 1\le i < j \le N\}$ is likely to have been generated from $\P_{\lambda}$ for a fixed $\lambda >0$ or from $\P_{0}$. Let us denote by $\bm{H}_0 : \Y \sim \P_{0}$ the null hypothesis and $\bm{H}_{\lambda}: \Y \sim \P_{\lambda}$ the alternative hypothesis.
Two formulations of this problem exist: one would like to construct a sequence of measurable tests $T: \R^{N(N-1)/2} \mapsto \{0,1\}$ that returns ``0" for ${\bm H}_0$ and ``1" for ${\bm H}_{\lambda}$, for which either
\begin{align}\label{strong_detection}
\lim_{N\to \infty} \P_{\lambda}(T(\Y) = 0) \vee \P_{0}(T(\Y) = 1) = 0,
\end{align}
or less stringently, the total mis-classification error, or risk,
\begin{equation}\label{misclassif_error}
\err(T) := \P_{\lambda} (T(\Y) = 0) + \P_{0} (T(\Y) = 1),
\end{equation}
is minimized among all possible tests $T$. 
\paragraph{Strong detection}
Using a second-moment argument (based on the computation of a truncated version of $\E L(\Y;\lambda)^2$), \cite{banks2017information} and \cite{perry2016optimality_annals} showed that $\P_{\lambda}$ and $\P_{0}$ are mutually contiguous when $\lambda < \lambda_0$, where the latter quantity equals $\lambda_c$ for some priors $P_{\xtt}$ while it is suboptimal for others (e.g., the sparse Rademacher case; see further discussion below).
It is easy to see that contiguity implies impossibility of strong detection since, for instance, if $\P_{0}(T(\Y)=1) \to 0$ then $\P_{\lambda}(T(\Y)=0) \to 1$.
Here we show that Theorem~\ref{central_limit_theorem} provides a more powerful approach to contiguity:
\begin{corollary}\label{contiguity}
Assume the prior $P_{\xtt}$ is centered, has unit variance and bounded support. Then for all $\lambda < \lambda_c$, $\P_{\lambda}$ and $\P_{0}$ are mutually contiguous.
\end{corollary}
\begin{proof}
A consequence of Theorem~\ref{central_limit_theorem} is that if
$\frac{\rmd \P_{\lambda}}{ \rmd \P_{0}} ~{\rightsquigarrow}~ U$
under $\P_{0}$ along some subsequence and for some random variable $U$, then by the continuous mapping theorem we necessarily have
$U = \exp \normal(-\mu,2\mu)$,
where $\mu = \frac{1}{4} \left(-\log(1-\lambda)-\lambda\right)$.
We have $\Pr(U>0) = 1$, and $\E U  = 1$. We now conclude using Le Cam's first lemma in both directions~\citep[Lemma 6.4 or Example 6.5,][]{vandervaart2000asymptotic}.
\end{proof}
This approach allows one to circumvent second-moment computations which are not guaranteed to be tight in general, and necessitate careful and prior-specific conditioning that truncates away problematic atypical events. On the other hand, we show that strong detection is possible above $\lambda_c$ (the proof is at the end of the paper): 
\begin{proposition}\label{orthogonality}
Let $\lambda > \lambda_c$. If $\Y \sim \P_{\lambda}$, then $\frac{1}{N} \log L(\Y;\lambda) >0$ with probability approaching one as $N\to +\infty$. On the other hand, if  $\Y \sim \P_{0}$ then $ \frac{1}{N} \log L(\Y;\lambda) \le 0$ with probability approaching one as $N \to +\infty$.
Therefore $\P_{\lambda}$ and $\P_{0}$ are mutually orthogonal above $\lambda_c$.
\end{proposition}

\begin{remark}It is tempting to believe that $ \overline{\lim}\frac{1}{N} \E_{\P_{0}}\log L(\Y;\lambda) < 0$ above $\lambda_c$ (the high-probability statement is then a consequence of concentration), but we do not know of a simple proof of this. One can show, following~\cite{guerra2003broken}, that there is a non-increasing sequence of thresholds $(\lambda_k)_{k\ge1}$---each one corresponding to the point where the so-called ``$k$-$\RSB$" interpolation bound dips below zero---such that the above limit is strictly negative above $\lambda_{\infty} = \lim \lambda_k$. By our contiguity argument, it is necessarily true that $\lambda_{\infty} \ge \lambda_c$. Equality would follow if one can show overlap convergence (the analogue of Theorem~\ref{convergence_fourth_moment_symmetric_paramagnetic} with $R_{1,2}$ replacing $R_{1,*}$) for all $\lambda < \lambda_{\infty}$ under the null model $\P_{0}$, but this goes beyond the scope of this paper.
\end{remark}

 We note that in the case of the sparse Rademacher prior, $P_{\xtt} = \frac{\rho}{2}\delta_{-1/\sqrt{\rho}} +   (1-\rho)\delta_{0} + \frac{\rho}{2}\delta_{+1/\sqrt{\rho}}$, we have $\lambda_c=1$ if $\rho \ge \rho^* \approx 0.092$ and $\lambda_c <1$ otherwise. Corollary~\ref{contiguity} and Proposition~\ref{orthogonality} exactly pin down the regime of contiguity, thus closing the gaps in the results of~\cite{banks2017information} and~\cite{perry2016optimality_annals}. 

\paragraph{Weak detection} We have seen that strong detection is possible if and only if $\lambda > \lambda_c$. It is then natural to ask whether weak detection is possible below $\lambda_c$; i.e., is it possible to test with accuracy \emph{better than that of a random guess} below the reconstruction threshold? The answer is \emph{yes}, and this is another consequence of Theorem~\ref{central_limit_theorem}. More precisely, the optimal test minimizing the risk~\eqref{misclassif_error} is the likelihood ratio test which rejects the null hypothesis $\bm{H}_0$ (i.e., returns ``1") if $L(\mtx{Y};\lambda) >1$, and its error is
\begin{equation}\label{optimal_error}
\err^*(\lambda) = \P_{\lambda} (L(\mtx{Y};\lambda) \le 1) + \P_{0} (L(\mtx{Y};\lambda) >1) = 1 - D_{\TV}(\P_\lambda,\P_0).
\end{equation}
One can readily deduce from Theorem~\ref{central_limit_theorem} the Type-I and Type-II errors of the likelihood ratio test. By symmetry of the means of the limiting Gaussians, the errors $\P_{0}(\log L(\mtx{Y}; \lambda) > 0) $ and $\P_{\lambda}(\log L(\mtx{Y}; \lambda) \le 0)$ converge to a common limit 
$\half\erfc\left(\frac{\sqrt{\mu}}{2}\right)$ 
for all $\lambda < \lambda_c$, where $\mu = \frac{1}{4} \left(-\log(1-\lambda)-\lambda\right)$ and $\erfc(x) = \frac{2}{\sqrt{\pi}} \int_{x}^\infty e^{-t^2}\rmd t$ is the complementary error function.
Therefore, one obtains the following formula for $\err^*(\lambda)$ and the total variation distance between $\P_{\lambda}$ and $\P_{0}$ (ploted in Figure~\ref{minimal_error_KL_TV}):
\begin{corollary}
For all $\lambda <\lambda_c$ (and $\sigma=+\infty$), we have
\begin{equation}\label{eq:total_variation}
\lim_{N\to \infty} \err^*(\lambda) = 1-\lim_{N\to \infty} D_{\TV}(\P_{\lambda},\P_{0}) = \erfc\left(\frac{1}{4}\sqrt{-\log(1-\lambda)-\lambda}\right).
\end{equation}
\end{corollary}
\begin{figure}
\centering
\includegraphics[width=7.5cm]{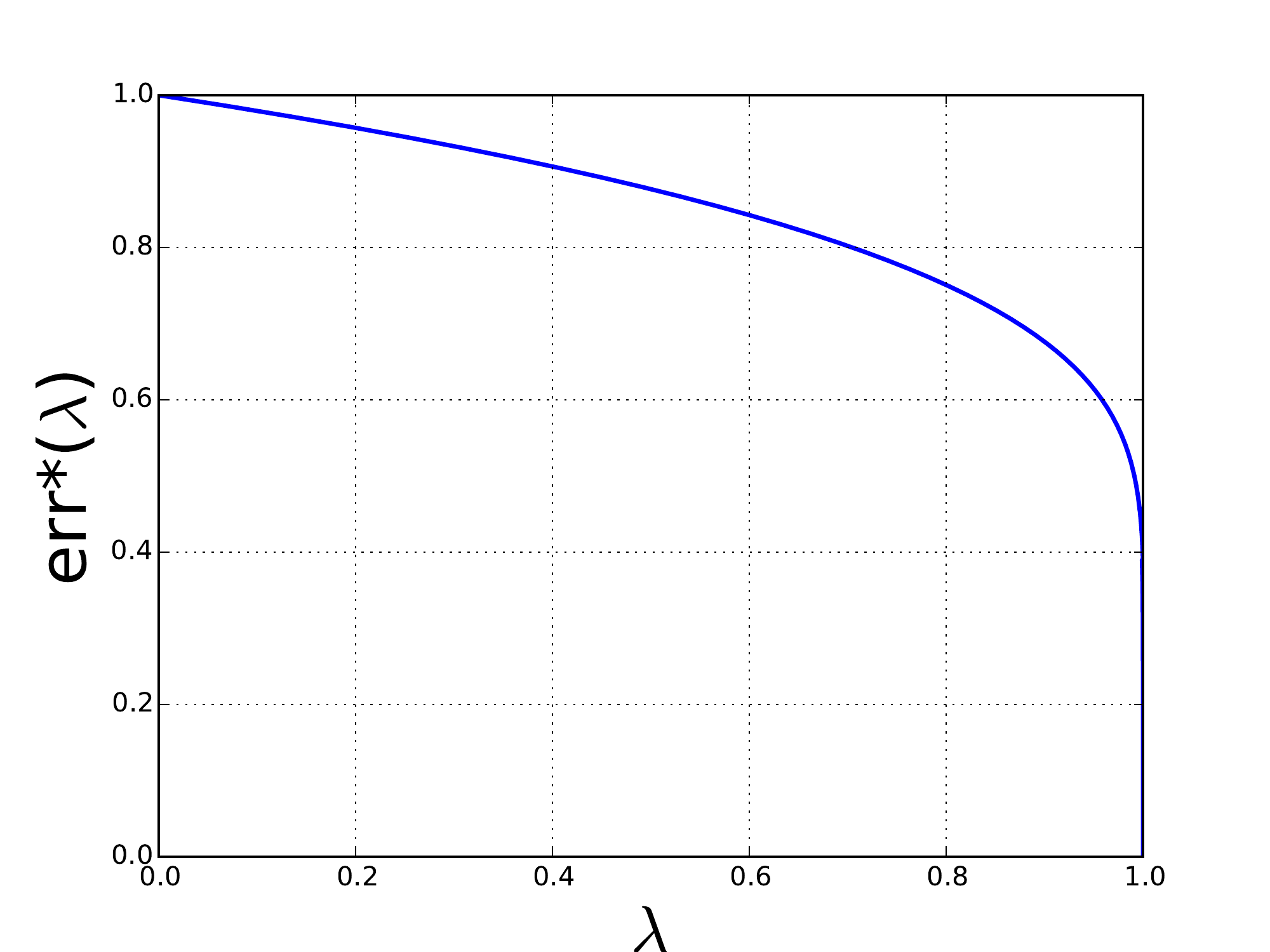}
\includegraphics[width=7.5cm]{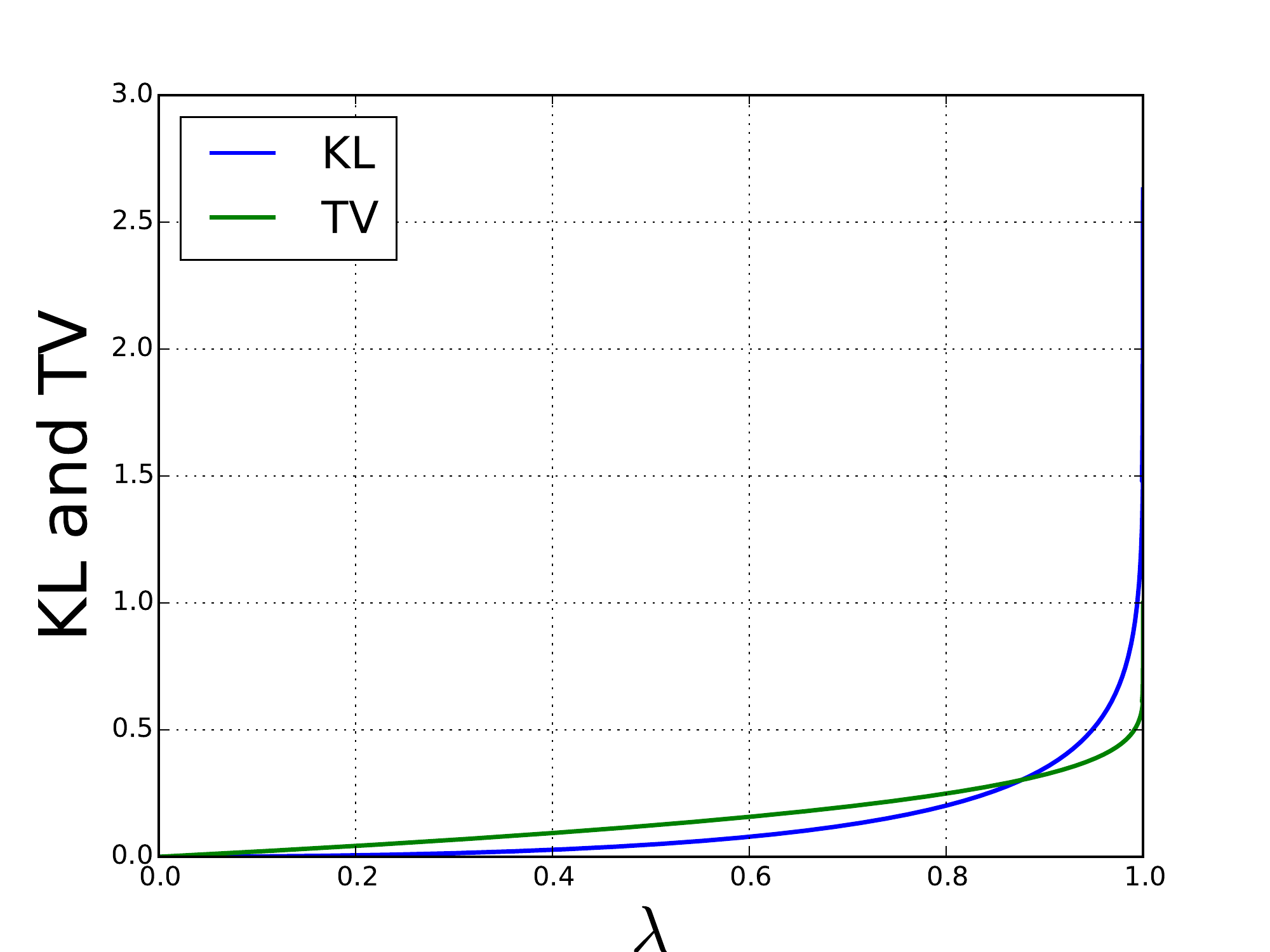}
\caption{Plots of formulas~\eqref{eq:total_variation} and~\eqref{eq:kul_divergence}.}
\protect\label{minimal_error_KL_TV}
\end{figure}
Moreover, the proof of Theorem~\ref{central_limit_theorem} allows us to obtain a formula for the $\KL$ divergence between $\P_{\lambda}$ and $\P_{0}$ below the reconstruction threshold $\lambda_c$ (see Figure~\ref{minimal_error_KL_TV}):
\begin{corollary}[of the proof]\label{cor_kul_divergence}
Assume the prior $P_{\xtt}$ is centered, is of unit variance and has bounded support (and $\sigma=+\infty$.) Then for all $\lambda <\lambda_c$,
\begin{equation}\label{eq:kul_divergence}
\lim_{N\to \infty} D_{\KL}(\P_\lambda,\P_0) = \frac{1}{4}\left(-\log\left(1-\lambda\right) - \lambda \right).
\end{equation}
\end{corollary}
Note that the above formulas are only valid up to $\lambda_c$.
When $\lambda_c <1$, $\TV$ and $\KL$ both witness an abrupt discontinuity at $\lambda_c$ to 1 and $\infty$ respectively.
When $\lambda_c=1$, then the behavior is more smooth with an asymptote at $1$.

\section{Replicas, overlaps, Gibbs measures and Nishimori}
\label{sxn:replicas_nishimori}
\subsection{Important notions} A crucial component of the proof of our main results is the understanding of the convergence of the overlap $\x^\top\x^*/N$, where $\x$ is drawn from $\P_{\lambda}(\cdot|\Y)$, to its limit $q^*(\lambda)$. By Bayes' rule, we see that
\begin{equation}\label{posterior}
\rmd\P_{\lambda} (\x| \Y) = \frac{ e^{- H(\x)} \rmd P_{\xtt}^{\otimes N}(\x)}{\int  e^{- H(\x)}  \rmd P_{\xtt}^{\otimes N}(\x)},
\end{equation}
where $H$ is the Hamiltonian (recall that $\sigma=+\infty$)
\begin{equation}\label{hamiltonian}
-H(\x) := \sum_{i < j} \sqrt{\frac{\lambda}{N}} Y_{ij}x_ix_j  -\frac{\lambda}{2N} x_i^2x_j^2.
\end{equation}
From the equations~\eqref{likelihood_ratio_explicit} and~\eqref{free_energy}, it is straightforward to see that
\[f_N = \frac{1}{N} \E_{\P_{\lambda}} \log~ \int e^{- H(\x)} \rmd P_{\xtt}^{\otimes N}(\x),\]
This provides another way of interpreting $f_N$ as the expected log-partition function (or normalizing constant) of the posterior $\P_{\lambda}(\cdot | \Y)$.
For an integer $n\ge 1$ and $f : (\R^{N})^{n+1} \mapsto \R$, we define the Gibbs average of $f$ w.r.t.\ $H$ as
 \begin{equation}\label{gibbs_average_one}
 \left\langle f(\x^{(1)},\cdots,\x^{(n)},\x^*)\right\rangle:= \frac{\int f(\x^{(1)},\cdots,\x^{(n)},\x^*) \prod_{l=1}^n e^{- H(\x^{(l)})} \rmd P_{\xtt}^{\otimes N}(\x^{(l)})}{\big(\int e^{- H(\x)}  \rmd P_{\xtt}^{\otimes N}(\x)\big)^n}.
 \end{equation}
This is simply the average of $f$ with respect to $\P_{\lambda}(\cdot | \Y)^{\otimes n}$.
 The variables $\x^{(l)}, l=1\cdots,n$, are called \emph{replicas}, and are interpreted as random variables drawn independently from the posterior. When $n=1$ we simply write $f(\x,\x^*)$ instead of $f(\x^{(1)},\x^*)$.  Throughout this paper, we use the following notation: for $l,l'=1,\cdots,n,*$, we let
 \[R_{l,l'} := \x^{(l)} \cdot \x^{(l')} = \frac{1}{N} \sum_{i=1}^N x_i^{(l)}x_i^{(l')}.\]

\subsection{The Nishimori property under $\P_{\lambda}$}
The fact that the Gibbs measure $\langle \cdot \rangle$ is a posterior distribution~\eqref{posterior} has far-reaching consequences. A crucial implication is that the $n+1$-tuples $(\x^{(1)},\cdots,\x^{(n+1)})$ and $(\x^{(1)},\cdots,\x^{(n)},\x^*)$ have the same law under $\E_{\P_{\lambda}}\langle \cdot \rangle$.
To see this, let us perform the following experiment:
\begin{enumerate}
\item Construct $\x^*\in \R^N$ by independently drawing its coordinates from $P_{\xtt}$.
\item Construct $\Y$ as $Y_{ij} = \sqrt{\frac{\lambda}{N}}x_i^*x_j^* + W_{ij}$, where $W_{ij} \sim \normal(0,1)$ are all independent for $i < j$. (Therefore, $\Y$ is distributed according to $\P_{\lambda}$.)
\item Draw $n+1$ independent random vectors $(\x^{(l)})_{l=1}^{n+1}$ from $\P_{\lambda}(\x \in \cdot |\Y)$.
\end{enumerate}
By the tower property of expectations~\citep[for a three-line proof, see Proposition 16,][]{lelarge2016fundamental}, the following equality of joint laws holds
\begin{equation}\label{nishimori_property_symmetric}
\left(\Y,\x^{(1)},\cdots,\x^{(n)},\x^{(n+1)}\right) \overset{\textup{d}}{=}  \left(\Y,\x^{(1)},\cdots,\x^{(n)},\x^{*}\right).
\end{equation}
This implies in particular that under the alternative $\P_{\lambda}$, the overlaps $R_{1,*}$ between a replica and the spike have the same distribution as the overlap $R_{1,2}$ between two replicas.
The latter is a very important property of the planted model $\P_{\lambda}$, which is usually named after~\cite{nishimori2001statistical} in spin-glass theory.
Property~\eqref{nishimori_property_symmetric} substantially simplifies important technical arguments that are otherwise very difficult to conduct under the null.
A recurring example in our context is the following: to prove the convergence of the overlap between two replicas, $\E\langle R_{1,2}^2\rangle  \to 0$, it suffices to prove $\E\langle R_{1,*}^2\rangle  \to 0$ since the two quantities are equal.

\section{Proof of LR fluctuations}
\label{sxn:proof_fluctuations}
In this section we prove Theorem~\ref{central_limit_theorem}.
It suffices to prove the fluctuations under one of the hypotheses. Fluctuations under the remaining one come for free as a consequence of Le Cam's third lemma \citep[][Theorem 6.6]{vandervaart2000asymptotic}.
We choose to treat the planted case $\Y \sim \P_{\lambda}$. The reason is that it is easier to deal with the planted model, due to the Nishimori property~\eqref{nishimori_property_symmetric}.

\subsection{Fluctuations under $\P_{\lambda}$}
In this section we prove Gaussian fluctuations of $\log L$ through the convergence of its characteristic function.
Let $\imnb^2=-1$ and $s \in \R$ be fixed. For $\lambda$ and $\Y \sim \P_{\lambda}$, let
\[\phi_N(\lambda) = \E_{\P_{\lambda}}\left[e^{\imnb s \log L(\Y;\lambda)}\right].\]

\begin{theorem}\label{characteristic_function_convergence}
For all $\lambda<\lambda_c$ and $s\in\R$, there exists a constant $K = K(\lambda,s) <\infty$ such that
\[\abs{\phi_N(\lambda) - e^{(\imnb s -s^2)\mu}}\le \frac{K}{\sqrt{N}},\]
where $\mu = \frac{1}{4}(-\log (1-\lambda) - \lambda)$.
\end{theorem}
The map $s \mapsto e^{(\imnb s -s^2)\mu}$ is the characteristic function of $\normal(\mu,2\mu)$.
\begin{lemma}\label{lemma_derivative_f}
For all $\lambda \ge 0$,
\begin{align} \label{derivative_f}
\phi_N'(\lambda) &= \frac{\imnb s - s^2}{4} \E\left[\left(N\langle R_{1,*}^2\rangle- \langle x_N^2x_N^{*2}\rangle\right)e^{\imnb s \log L}\right].
\end{align}
\end{lemma}
\begin{proof}
By differentiation with respect to $\lambda$ we obtain
\[\phi_N'(\lambda) = \imnb s \E[(\frac{\rmd }{\rmd \lambda} \log L) e^{\imnb s \log L}] 
= \imnb s \E[\langle -\frac{\rmd }{\rmd \lambda}H(\x) \rangle e^{\imnb s \log L}], \]
where the Hamiltonian $H$ is given in~\eqref{hamiltonian}. Since $\Y \sim \P_{\lambda}$, we can write more explicitly
$-H(\x) = \sum_{i < j} \sqrt{\frac{\lambda}{N}} W_{ij}x_ix_j + \frac{\lambda}{N} x_ix_jx_i^*x_j^* -\frac{\lambda}{2N} x_i^2x_j^2$.
Therefore 
\begin{align}\label{derivative_f_expanded}
\phi_N'(\lambda) &= \imnb s \sum_{i < j} \frac{1}{2\sqrt{\lambda N}}\E[\langle W_{ij}  x_ix_j\rangle e^{\imnb s \log L}] -\frac{1}{2N}  \E[\langle x_i^2x_j^2\rangle e^{\imnb s \log L}]\\
&~+  \imnb s \sum_{i < j} \frac{1}{N}\E[\langle x_ix_jx_i^*x_j^*\rangle e^{\imnb s \log L}].\nonumber
\end{align}
Now we perform Gaussian integration by parts with respect to each variable $W_{ij}$ and obtain  
\begin{align*}
\frac{1}{2\sqrt{\lambda N}}\E[\langle W_{ij}  x_ix_j\rangle  e^{\imnb s \log L}] &= \frac{1}{2N}\E[\langle  x_i^2x_j^2\rangle e^{\imnb s \log L}] - \frac{1}{2N}\E[\langle  x_ix_j\rangle^2 e^{\imnb s \log L}] \\
&~+\frac{\imnb s}{2N}\E[\langle  x_ix_j\rangle^2 e^{\imnb s \log L}].
\end{align*}
Plugging this into~\eqref{derivative_f_expanded} and rearranging, we obtain 
\begin{align}\label{derivative_phi}
\phi_N'(\lambda) &= -\frac{\imnb s + s^2}{4} \E[\left(N\langle R_{1,2}^2\rangle - \langle x_N^2\rangle^2\right) e^{\imnb s \log L}] \\
&~+ \frac{\imnb s}{2}\E[\left(N\langle R_{1,*}^2\rangle - \langle x_N^2x_N^{*2}\rangle\right) e^{\imnb s \log L}]. \nonumber
\end{align}
Since we are under the planted model $\P_{\lambda}$ and $e^{\imnb s \log L}$ depends only on $\Y$, we can use the Nishimori property~\eqref{nishimori_property_symmetric} to replace $R_{1,2}$ and $x_N^{(1)}x_N^{(2)}$ by $R_{1,*}$ and $x_Nx_N^{*}$ respectively in the first term of~\eqref{derivative_phi}.
\end{proof}

The derivative involves the average $\E[(N\langle R_{1,*}^2\rangle - \langle x_N^2x_N^{*2}\rangle)e^{\imnb s \log L}]$.
A crucial step in the argument is to show that $e^{\imnb s \log L}$ and its pre-factor in the above expression are asymptotically independent, so that one can split the expectation of the product into the product of the expectations.
More precisely, one should expect the quantities $N\langle R_{1,*}^2\rangle$ and $\langle x_N^2x_N^{*2}\rangle$ to tightly concentrate about some deterministic values when $\lambda < \lambda_c$, such that the second expectation in~\eqref{derivative_f} is a multiple of $\E[e^{\imnb s \log L}] = \phi_N(\lambda)$. We will then be left with a simple differential equation whose solution is $s \mapsto e^{(\imnb s -s^2)\mu}$.

\begin{proposition}\label{delicate_overlap_convergence}
For all $\lambda<\lambda_c$ and $s\in \R$, there exists $K = K(\lambda,s)<\infty$ such that 
\[\E\left[\left(N\langle R_{1,*}^2\rangle - \langle x_N^2x_N^{*2}\rangle\right)e^{\imnb s \log L}\right] = \frac{\lambda}{1-\lambda}\E\left[ e^{\imnb s \log L}\right] + \delta,\]
 where $|\delta| \le K(s,\lambda)/\sqrt{N}$.
\end{proposition}

From here, we can prove the convergence of $\phi_N$ by integrating the differential equation given in Lemma~\ref{lemma_derivative_f}.

\begin{proofof}{Theorem~\ref{characteristic_function_convergence}}
Plugging the result of Proposition~\ref{delicate_overlap_convergence} into Lemma~\ref{lemma_derivative_f} yields
\[\phi_N'(\lambda) =\frac{\imnb s - s^2}{4} \frac{\lambda}{1-\lambda}\phi_N(\lambda)  + \delta,\]
where $|\delta| \le K(s,\lambda)/\sqrt{N}$.
Since $\phi_N(0) = 1$ and the primitive of $\lambda \mapsto \frac{\lambda}{1-\lambda}$ is $\lambda \mapsto -\lambda - \log(1-\lambda)$, integrating w.r.t.\ $\lambda$ yields the result.
\end{proofof}

\begin{proofof}{Corollary~\ref{cor_kul_divergence}}
We prove the convergence of $D_{\KL}(\P_{\lambda},\P_{0})$. By differentiation and use of the Nishimori property~\eqref{nishimori_property_symmetric}, we have
\begin{align*}
\frac{\rmd}{\rmd \lambda}\E_{\P_{\lambda}} \log L(\Y;\lambda) &= -\frac{1}{4}\E[(N\langle R_{1,2}^2\rangle - \langle x_N^2\rangle^2)]+\frac{1}{2}\E[(N\langle R_{1,*}^2\rangle - \langle x_N^2x_N^{*2}\rangle)]\\
 &~= \frac{1}{4}\E[(N\langle R_{1,*}^2\rangle - \langle x_N^2x_N^{*2}\rangle)].
\end{align*}
Now we use Proposition~\ref{delicate_overlap_convergence} with $s=0$, and integrate w.r.t.\ $\lambda$ to conclude.
\end{proofof}

It remains to prove Proposition~\ref{delicate_overlap_convergence}.
This will require the deployment of techniques from the theory of mean-field spin glasses.

\subsection{Sketch of proof of Proposition~\ref{delicate_overlap_convergence}}
The idea is to show self-consistency relations among the quantities of interest.
Namely, we will prove that for all $\lambda <1$,
\begin{equation}\label{first_fundamental_bound}
N \E\left[\langle R_{1,*}^2\rangle e^{\imnb s \log L}\right] = \frac{1}{1-\lambda}\E\left[\langle x_N^2x_N^{*2}\rangle e^{\imnb s \log L}\right] + \delta,
\end{equation}
and
\begin{equation}\label{second_fundamental_bound}
\E\left[\langle x_N^2x_N^{*2}\rangle e^{\imnb s \log L}\right] = \E\left[e^{\imnb s \log L}\right] + \delta,
\end{equation}
where in both cases
\[|\delta| \le K(\lambda)N\E\left\langle |R_{1,*}|^3\right\rangle.\]
Next, we need to prove the convergence of the third moment of the overlap $R_{1,*}$ under $\E \langle \cdot \rangle$ at an optimal rate of $\bigo(1/N^{3/2})$:
\begin{theorem}\label{convergence_fourth_moment_symmetric_paramagnetic}
For all $\lambda<\lambda_c$, there exists a constant $K=K(\lambda) <\infty $ such that
\[\E \langle R_{1,*}^4\rangle \le \frac{K}{N^2}.\]
\end{theorem}
This will allow us to conclude that $|\delta| \le K(\lambda)/\sqrt{N}$.
It is interesting to note that while the self-consistent (or cavity) equations~\eqref{first_fundamental_bound} and~\eqref{second_fundamental_bound} hold for all $\lambda<1$, the convergence of the overlap towards zero is only true up to $\lambda_c$.

\section{Proof of asymptotic decoupling}
\label{sxn:asymptotic_decoupling}
We proceed to the proof of Proposition~\ref{delicate_overlap_convergence}.
As explained earlier, the argument is in two stages. We first prove~\eqref{first_fundamental_bound} then~\eqref{second_fundamental_bound}. 

\subsection{Preliminary bounds}
We make repeated use of interpolation arguments in our proofs. We state here a few elementary lemmas that we will invoke several times.
We denote the overlaps between replicas where the last variable $x_N$ is deleted by a superscript ``-":
\[R^{-}_{l,l'} = \frac{1}{N}\sum_{i=1}^{N-1} x_i^{(l)}x_i^{(l')}.\]
Let $\{H_t: t\in [0,1]\}$ be a family of interpolating Hamiltonians. We let $\langle \cdot \rangle_t$ denote the corresponding Gibbs average, similarly to~\eqref{gibbs_average_one}.
Following Talagrand's notation, we write
\[\nu_t(f) := \E\langle f\rangle_{t},\]
for a generic function $f$ of $n$ replicas $\x^{(l)}$, $l=1,\cdots,n$. We abbreviate $\nu_1$ by $\nu$.
The main tool we use is the following interpolation that isolates the last variable $x_N$ from the rest of the system:
\begin{align}\label{interpolating_hamiltonian_symmetric_paramagnetic}
-H_t(\x) &:= \sum_{1\le i < j \le N-1}  \sqrt{\frac{\lambda }{N}} W_{ij}x_ix_j + \frac{\lambda }{N}x_ix_i^*x_jx_j^* -\frac{\lambda }{2N} x_i^2x_j^2\\
&~~+\sum_{i =1}^{N-1} \sqrt{\frac{\lambda t}{N}} W_{iN}x_i x_N + \frac{\lambda t}{N} x_ix_i^* x_Nx_N^* - \frac{\lambda t}{2N} x_i^2 x_N^2.\nonumber
\end{align}
At $t=1$ we have $H_{t} = H$, and at $t=0$ the variable $x_N$ decouples from the rest of the variables. Moreover, the Nishimori property~\eqref{nishimori_property_symmetric} is still valid under $\langle \cdot \rangle_t$: the last column of $\Y$ simply becomes $\big(\sqrt{\frac{\lambda t}{N}}x_i^*x_N^* + W_{iN}\big)_{i=1}^{N-1}$.   
\begin{lemma}\label{gibbs_derivative}
let $f$ be a function of $n$ replicas $\x^{(1)},\cdots,\x^{(n)}$ and $\x^*$. Then
\begin{align*}
\nu_t'(f) &= \frac{\lambda }{2} \sum_{1\le l\neq l' \le n} \nu_t(R^{-}_{l,l'} y^{(l)}y^{(l')}f)
- \lambda  n \sum_{l=1}^n \nu_t(R^{-}_{l,n+1} y^{(l)}y^{(n+1)}f) \\
&~~+ \lambda  n \sum_{l=1}^n \nu_t(R^{-}_{l,*} y^{(l)}y^{*}f)
- \lambda  n  \nu_s(R^{-}_{n+1,*} y^{(n+1)}y^{*}f) \\
&~~+  \lambda  \frac{n(n+1)}{2} \nu_t(R^{-}_{n+1,n+2} y^{(n+1)}y^{(n+2)}f),
\end{align*}
where we have written $y = x_N$.
\end{lemma}
\begin{proof}
The computation relies on Gaussian integration by parts. See~\cite{talagrand2011mean1}, Lemma 1.6.3, for the details of a similar computation.
\end{proof}

\begin{lemma}\label{bound_gibbs_derivative}
If $f$ is a bounded nonnegative function, then for all $t \in [0,1]$,
\[\nu_t(f) \le K(\lambda ,n) \nu(f).\]
\end{lemma}

\begin{proof}
Since the variables and the overlaps are all bounded, using Lemma~\ref{gibbs_derivative} we have for all $t\in [0,1]$
\[|\nu_t'(f)| \le K(\lambda ,n) \nu_t(f).\]
Then we conclude using Gr\"onwall's lemma.
\end{proof}

\subsection{The cavity method}
In its essence, the cavity method amounts to removing one variable from the system---in a manner
akin to leave-one-out methods in statistics---and analyzing the influence of the remaining variables
on the variable that has been removed. It was initially introduced to solve certain models of spin
glasses~\citep{mezard1990spin}, and was developed into a rigorous probabilistic theory by
\cite{talagrand2011mean1,talagrand2011mean2}.  To make use of the cavity method, we isolate
the $N$th variable from the rest (without loss of generality, by symmetry among the variables $x_i$)
and compute:
\[\E\left[\left(N\langle R_{1,*}^2\rangle - \langle x_N^2x_N^{*2}\rangle\right)e^{\imnb s \log L}\right] = N\E\left[\left\langle x_Nx_N^*R_{1,*}^-\right\rangle e^{\imnb s \log L}\right].\]
Let
\[X(t):= \exp\left(\imnb s \log \int e^{-H_t(\x)} \rmd P_{\xtt}^{\otimes N}(\x)\right),\]
where $H_t$ is defined in~\eqref{interpolating_hamiltonian_symmetric_paramagnetic}.
Note that we have $X(1) = e^{\imnb s \log L}$.  We now consider the interpolative function
\[\varphi(t) :=  N\E\left[\left\langle x_Nx_N^*R_{1,*}^-\right\rangle_tX(t)\right].\]
Our strategy is approximate $\varphi(1)$ by $\varphi(0)+\varphi'(0)$ via a Taylor expansion, which requires is to control
the second derivative $\varphi''$. Notice that since the last variable decouples from the rest of the system at $t=0$, we have
\begin{align*}
\varphi(0) &= N\E\left[\langle x_Nx_N^{*}\rangle_0\right] \cdot \E\left[\langle R_{1,*}^-\rangle_0X(0)\right]\\
&= N\E_{P_{\xtt}}[X]^2 \cdot \E\left[\langle R_{1,*}^-\rangle_0X(0)\right] = 0.
\end{align*}
The latter equality holds since $P_{\xtt}$ is centered. 
Next, a bit of algebra (similarly to Lemma~\ref{gibbs_derivative}) shows that the derivative $\varphi'(t)$ is a linear combination of terms of the form
\begin{equation}\label{first_derivative_order_two}
\lambda N \E\left[\left\langle x_Nx_N^{*}x_N^{(a)}x_N^{(b)} R^-_{1,*}R^-_{a,b} \right\rangle_t X(t)\right],
\end{equation}
where $(a,b) \in \{(1,*),(2,*),(1,2),(2,3)\}$.
We see that at $t=0$, if the above expression involves a variable $x_N^{(a)}$ of degree 1 then this term vanishes. Therefore the only remaining term is the one where $(a,b) = (1,*)$. Therefore
\begin{align}\label{second_derivative_at_0}
\varphi'(0) &= \lambda N  \E\left[\langle x_N^2x_N^{*2}\rangle_0\right] \cdot \E \left[\langle (R_{1,*}^-)^2 \rangle_0 X(0) \right] \\
&= \lambda N  \E_{P_{\xtt}}[X^2]^2 \cdot \E \left[\langle (R_{1,*}^-)^2 \rangle_0 X(0) \right]\nonumber\\
&= \lambda N  \E \left[\langle (R_{1,*}^-)^2 \rangle_0 X(0) \right].\nonumber
\end{align}
The last equality holds since $P_{\xtt}$ has unit variance.
Now we turn to $\varphi''(t)$. Taking another derivative generates monomials of degree three in the overlaps and the last variable, so $\varphi''(t)$ is a linear combination of terms
\begin{equation}\label{second_derivative_order_three}
\lambda^2 N \E\left[\left\langle x_Nx_N^*x_N^{(a)}x_N^{(b)}x_N^{(c)}x_N^{(d)}R^-_{1,2}R^-_{a,b}R^-_{c,d} \right\rangle_t X(t)\right],
\end{equation}
where $(a,b,c,d)$ range over a finite set of combinations. 
Our goal is to bound the second derivative independently of $t$, so that we are able to use the Taylor approximation
\begin{equation}\label{taylor_number_2}
\abs{\varphi(1) -\varphi(0)-\varphi'(0)} \le \sup_{0\le t \le 1}\abs{\varphi''(t)}.
\end{equation}
Since $P_{\xtt}$ has bounded support and $|X(t)|=1$, H\"older's inequality and the Nishimori property imply that~\eqref{second_derivative_order_three} is bounded in modulus by
\begin{align*}
\lambda N  K \E\left[\left\langle \abs{R^-_{1,2}R^-_{a,b}R^-_{c,d}}\right\rangle_t\right]
 \le \lambda N  K\E\left[\left\langle |R^-_{1,*}|^{3}\right\rangle_t\right]^{1/3}.
\end{align*}
Using Lemma~\ref{bound_gibbs_derivative} and the convergence of the fourth moment, Theorem~\ref{convergence_fourth_moment_symmetric_paramagnetic}, we have
\[\E\left\langle |R^-_{1,*}|^{3}\right\rangle_t \le K(\lambda)\E\left\langle (R^-_{1,*})^{4}\right\rangle^{3/4} \le \frac{K(\lambda)}{N^{3/2}}.\]

Therefore by the above estimates we have
\begin{equation}\label{second_derivative_bound}
\sup_{0\le t \le 1}|\varphi''(t)| \le \frac{K(\lambda)}{\sqrt{N}}.
\end{equation}
Now, our next goal is to prove
\begin{equation}\label{first_derivative_approx}
\abs{ \varphi'(0) - \lambda N \E \left[\langle R_{1,*}^2 \rangle e^{\imnb s \log L} \right]} \le\frac{K(\lambda)}{\sqrt{N}}.
\end{equation}
We consider the function
\[\psi(t) := \lambda N \E \left[\left\langle (R_{1,*}^-)^2 \right\rangle_t X(t) \right].\]
Observe that~\eqref{second_derivative_at_0} tells us that $\psi(0)=\varphi'(0)$. On the other hand,
\begin{align*}
\abs{\psi(1) -  \lambda N \E \left[\langle R_{1,*}^2 \rangle e^{\imnb s \log L}\right]} \le 2\lambda \E\left\langle \abs{R_{1,*}^- x_Nx_N^{*}}\right\rangle
+ \frac{\lambda}{N} \E\left\langle (x_Nx^{*}_N)^2\right\rangle.
\end{align*}
By boundedness of the prior, the first term in the RHS is bounded by
\[K(\lambda) \E\langle |R_{1,*}^-|\rangle \le K(\lambda)/\sqrt{N},\]
and the second term is bounded by $K(\lambda)/N$. So it suffices to show that
\[\sup_{0\le t \le 1}|\psi'(t)| \le \frac{K(\lambda)}{\sqrt{N}}.\]
Similarly to $\varphi$, the derivative of $\psi$ is a sum of terms of the form
\[\lambda^2 N \E\left[\left\langle x_N^{(a)}x_N^{(b)}(R^-_{1,*})^2R^-_{a,b}\right\rangle_t X(t)\right].\]
It is clear that the same method used to bound $\varphi''$ (the generic term of which is~\eqref{second_derivative_order_three}) also works in this case, so we obtain the desired bound on $\psi'$.
Finally, using~\eqref{taylor_number_2}, \eqref{second_derivative_bound} and \eqref{first_derivative_approx}, we obtain
\[N\E\left[\langle R_{1,*}^2\rangle e^{\imnb s \log L}\right] - \E\left[\left\langle x_N^2x^{*2}_N\right\rangle e^{\imnb s \log L}\right] 
= \lambda N \E \left[\langle R_{1,*}^2 \rangle e^{\imnb s \log L} \right] + \delta,\]
where $|\delta| \le K(\lambda)/\sqrt{N}$. This is equivalent to~\eqref{first_fundamental_bound} and closes the first stage of the argument.
Now we need to show that
\[\E\left[\left\langle x_N^2x_N^{*2}\right\rangle  e^{\imnb s \log L}\right] = \E\left[ e^{\imnb s \log L}\right] + \delta.\]
We similarly consider the function $\psi(t) = \E\left[\left\langle  x_N^2x_N^{*2}\right\rangle_t X(t)\right]$.
We have
\[\psi(0) = \E\left[\left\langle  x_N^2x_N^{*2}\right\rangle_0\right] \cdot \E\left[X(0)\right] =  \E_{P_{\xtt}}[X^2]^2 \cdot \E\left[X(0)\right] = \E\left[X(0)\right].\]
The derivative of $\psi$ is a sum of term of the form
\[\lambda \E\left[\left\langle  x_N^2x_N^{*2} x_N^{(a)}x_N^{(b)}R^-_{a,b}\right\rangle_t X(t)\right].\]
By our earlier argument, $|\psi'(t)| \le K(\lambda)/\sqrt{N}$ for all $t$, so that
\[\abs{\psi(1) - \E\left[X(0)\right]} \le \frac{K(\lambda)}{\sqrt{N}}.\]
It remains to show that $\big|\E\left[X(0)\right] - \E\left[X(1)\right]\big| \le K/\sqrt{N}$, and this is
done in exactly the same way: by bounding the derivative of $t \mapsto \E\left[X(t)\right]$. 
This yields~\eqref{second_fundamental_bound} and concludes the proof.

\section{Overlap convergence}
\label{sxn:overlap_convergence}
In this section we prove Theorem~\ref{convergence_fourth_moment_symmetric_paramagnetic} on the convergence of the overlaps to zero under $\P_{\lambda}$, and below $\lambda_c$.
At a high level, we will first prove that the overlap $R_{1,*}$ converges in probability to zero under $\E \langle \cdot \rangle$: for all $\epsilon >0$,
\begin{equation}\label{crude_bound}
\E\langle \indi\{|R_{1,*}| \ge \epsilon\}\rangle \le Ke^{-cN}.
\end{equation}
This will be achieved via two interpolation bounds combined with concentration of measure. The way the argument works is roughly as follows: for a fixed $q$ we have  
\begin{align*}
\E\langle \indi\{R_{1,*} \simeq q \}\rangle  &= \E \frac{\int \indi\{R_{1,*}\simeq q\}e^{-H(\x)}\rmd P_{\xtt}^{\otimes N}(\x)}{\int e^{-H(\x)}\rmd P_{\xtt}^{\otimes N}(\x)}\\
&=\E \frac{\exp \{N \times \frac{1}{N}\log\int \indi\{R_{1,*}\simeq q\}e^{-H(\x)}\rmd P_{\xtt}^{\otimes N}(\x)\}}{\exp \{N \times \frac{1}{N}\log\int e^{-H(\x)}\rmd P_{\xtt}^{\otimes N}(\x)\}}.
\end{align*}
We invoke concentration-of-measure arguments to show that the logarithmic terms in the numerator and the denominator are close to their expectations, hence we obtain
\[\E\langle \indi\{R_{1,*} \simeq q \}\rangle \simeq \exp\{N(f_N(q) - f_N)\},\]
where $f_N(q)=\frac{1}{N}\E\log\int \indi\{R_{1,*}\simeq q\}e^{-H(\x)}\rmd P_{\xtt}^{\otimes N}(\x)$ and $f_N$ is the unconstrained free energy (with no indicator). It is now apparent that $R_{1,*}$ is exponentially unlikely to take values $q$ such that $f_N(q) < f_N$. It remains to lower bound $f_N$ and upper bound $f_N(q)$ by quantities that preserve a strict inequality for all $q \neq 0$. These quantities will naturally be the replica-symmetric formula $\phi_{\RS}(\lambda)$ and the replica-symmetric potential $F(\lambda,q)$ respectively, and the proof relies on Guerra's interpolation method. 

Next, this convergence in probability is boosted to a statement of convergence of the second moment:
$\E\langle R_{1,*}^2\rangle \le \frac{K}{N}$,
which is in turn boosted to a statement of convergence of the fourth moment:
 $\E\langle R_{1,*}^4\rangle \le \frac{K}{N^2}$.
 The apparent recursive nature of this argument is a feature of the cavity method: one can control higher-order quantities once one knows how to control low-order ones and control certain error terms. We now present the interpolation bounds and then prove~\eqref{crude_bound}. The cavity arguments which allow us to convert this to convergence of moments are presented in Appendices~\ref{sxn:convergence_second_moment} and~\ref{sxn:convergence_fourth_moment}, since they are very similar to the arguments already presented in Section~\ref{sxn:asymptotic_decoupling}.  

 \subsection{Guerra's interpolation bound}
We present the interpolation method of~\cite{guerra2001sum}; a main tool in our arguments. 
\begin{proposition}\label{lower_bound_f} 
Recall $f_N = \frac{1}{N} \E_{\P_{\lambda}}\log L(\Y;\lambda)$. For all $\lambda \ge 0$, there exist $K>0$ such that
\[f_N \ge \sup_{q \ge 0} F(\lambda,q) - \frac{K}{N} = \phi_{\RS}(\lambda) - \frac{K}{N}.\]
\end{proposition}
\begin{proof}
Consider the family of interpolating Hamiltonians
\begin{align}\label{interpolating_hamiltonian}
-H_t(\vct{x}) &:=  \sum_{i < j} \sqrt{\frac{t\lambda}{N}}W_{ij}x_ix_j + \frac{t\lambda}{N}x_ix_i^*x_jx_j^* -\frac{t\lambda}{2N}x_i^2x_j^2 \\
&~~~+  \sum_{i=1}^N \sqrt{(1-t)r}z_{i}x_i +  (1-t)r x_ix_i^*-\frac{(1-t)r}{2} x_i^2 ,\nonumber
\end{align}
where the $z_i$'s are i.i.d.\ standard Gaussian r.v.'s independent of everything else, and $r = \lambda q^*(\lambda)$. We similarly define the Gibbs average $\langle \cdot \rangle_t$ as in~\eqref{gibbs_average_one} where $H$ is replaced by $H_t$. Note that the Nishimori property~\eqref{nishimori_property_symmetric} is preserved under $\langle \cdot \rangle_t$ for all $t \in [0,1]$. Indeed, the interpolation is constructed in such a way that $\langle \cdot \rangle_t$ is the posterior distribution of the signal $\vct{x}^*$ given the augmented set of observations
\begin{align}\label{augmented_observations}
\begin{cases}
Y_{ij} &= \sqrt{\frac{t\lambda}{N}} x^*_ix^*_j + W_{ij}, \quad 1 \le i < j \le N, \\
y_i &= \sqrt{(1-t)r}x^*_i + z_i, \quad 1 \le i \le N,
\end{cases}
\end{align}
which can be interpreted as having side information about $\vct{x}^*$ from a scalar Gaussian channel with $r= \lambda q^*(\lambda)$.
Now we consider the interpolating free energy
\begin{equation}\label{interpolating_free_energy}
\varphi(t) :=  \frac{1}{N} \E \log \int e^{- H_t(\vct{x})} \rmd P_{\xtt}^{\otimes N}(\vct{x}).
\end{equation}
We see that $\varphi(1) = f_N$ and $\varphi(0) = \psi(\lambda q)$. This function is differentiable in $t$, and
by differentiation, we have
\begin{align*}
\varphi'(t) &= \frac{1}{N} \E \left \langle -\frac{\rmd H_t(\vct{x})}{\rmd t}\right\rangle_t\\
&= \frac{1}{N} \E \left\langle -\frac{\lambda}{2N}\sum_{i < j} x_i^2x_j^2 + \half \sqrt{\frac{\lambda}{t N}}\sum_{i < j} W_{ij}x_ix_j + \frac{\lambda}{N}\sum_{i < j}x_ix_i^*x_jx_j^* \right\rangle_t \\
&~~~+\frac{1}{N}\E \left\langle \frac{\lambda q}{2}\sum_{i=1}^N x_i^2 -  \half\sqrt{\frac{\lambda q}{1-t}}\sum_{i=1}^N z_{i}x_i - \lambda q\sum_{i=1}^N x_ix_i^*\right\rangle_t.
\end{align*}
Now we use Gaussian integration by parts to eliminate the variables $W_{ij}$ and $z_i$. The details of this computation are explained extensively in many sources~\citep[see, e.g.,][]{talagrand2011mean1,krzakala2016mutual,lelarge2016fundamental}. We get
\begin{align*}
\varphi'(t) &= -\frac{\lambda}{2N^2} \E\left\langle \sum_{i < j} x_i^{(1)} x_j^{(1)}x_i^{(2)} x_j^{(2)}\right\rangle_t + \frac{\lambda}{N^2}  \E\left\langle  \sum_{i < j}  x_i x_i^{*} x_j x_j^{*} \right\rangle_t \\
&~~~+\frac{\lambda q}{2N} \E \left\langle \sum_{i=1}^N x_i^{(1)} x_i^{(2)} \right\rangle_t - \frac{\lambda q}{N} \E \left\langle \sum_{i=1}^N x_ix_i^*\right\rangle_t.
\end{align*}
Completing the squares yields
\begin{align}\label{derivative_before_nishimori}
\varphi'(t) &= -\frac{\lambda}{4} \E\left\langle (\vct{x}^{(1)}\cdot\vct{x}^{(2)}-q)^2\right\rangle_t + \frac{\lambda}{4} q^2 + \frac{\lambda}{4N^2} \sum_{i=1}^N \E\left\langle {x_i^{(1)}}^2{x_i^{(2)}}^2\right\rangle_t \\
&~~~+\frac{\lambda}{2} \E\left\langle (\vct{x}\cdot\vct{x}^*-q)^2\right\rangle_t - \frac{\lambda}{2} q^2 - \frac{\lambda}{2N^2} \sum_{i=1}^N \E\left\langle {x_i}^2{x_i^{*}}^2\right\rangle_t.\nonumber
\end{align}
The first line in the above expression involves overlaps between two independent replicas, while the second one involves overlaps between one replica and the planted solution. Using the Nishimori property, the derivative of $\varphi$ can be written as
\begin{equation}\label{free_energy_derivative}
\varphi'(t) = \frac{\lambda}{4} \E\left\langle (R_{1,*}-q)^2 \right\rangle_t -\frac{\lambda}{4} q^2 - \frac{\lambda}{4N} \E\left\langle {x_N}^2{x_N^{*}}^2\right\rangle_t.
\end{equation}
The last term follows by symmetry between variables. 
We finish the argument by noting that $\E\left\langle (R_{1,*}-q)^2\right\rangle_t \ge 0$, and the product $ {x_N}^2 {x_N^*}^2$ is bounded, we then integrate with respect to time to obtain the result. 
\end{proof}

\subsection{Guerra's interpolation at fixed overlap}
\label{sxn:main_estimate}
Let us first introduce the so-called \emph{Franz-Parisi (FP) potential}~\citep{franz1995recipes,franz1998effective}.
For $\x^* \in \R^N$ fixed, $m \in \R\setminus \{0\}$ and $\epsilon>0$ define the set
\begin{align*}
A =
\begin{cases}
R_{1,*} \in [m,m+\epsilon) & \mbox{if } m >0, \\
R_{1,*} \in (m-\epsilon,m] & \mbox{if } m <0.
\end{cases}
\end{align*}
Now define the FP potential as
\begin{equation}\label{free_energy_fixed_overlap}
\Phi_{\epsilon}(m,\x^*) := \frac{1}{N}\E_{\W}\log \int \indi\{\x \in A\} e^{-H(\x)}  \rmd P_{\xtt}^{\otimes N}(\x),
\end{equation}
where the expectation is only over the Gaussian disorder $\W$.
This is the free energy of a subsystem of configurations having an overlap close to a fixed value $m$ with the planted signal $\x^*$.

For $r \ge 0$ and $s\in \R$, we let
\begin{equation}\label{psi_hat}
\widehat{\psi}(r,s) := \E_{z} \log \int \exp\left(\sqrt{r}zx + s x - \frac{r}{2} x^2\right) \rmd P_{\xtt}(x).
\end{equation}
and
\begin{align}\label{psi_bar}
\widebar{\psi}(r,s) &:= \E_{x^*}\widehat{\psi}(r,sx^*)\nonumber\\
&=\E_{x^*,z} \log \int \exp\left(\sqrt{r}zx + s xx^* - \frac{r}{2} x^2\right) \rmd P_{\xtt}(x).
\end{align}
We see that $\widebar{\psi}(r,r) = \psi(r)$, but unlike $\psi$, the function $\widebar{\psi}$ does not have an interpretation as the $\KL$ between two distributions.
The next lemma states a key property of this function that will be useful later on (see Appendix~\ref{sxn:proof_of_lemma_six} for the proof):
\begin{lemma}\label{asymmetry_lemma}
  For all $r \ge 0$, $\widebar{\psi}(r,-r) \le \widebar{\psi}(r,r)$.
\end{lemma}
Additionally, for $\x^* \in \R^N$ fixed, we define the function
\[\widehat{F}(\lambda,m,q) := \frac{1}{N}\sum_{i=1}^N\widehat{\psi}(\lambda q,\lambda m x_i^*) - \frac{\lambda}{2} m^2+\frac{\lambda}{4} q^2.\]
Recall that $\E_{x^*} \widehat{F}(\lambda,q,q)$ is the $\RS$ potential $F(\lambda,q)$ from~\eqref{RS_potential}.
\begin{proposition} \label{fixed_overlap_upper_bound_paramagnetic}
Fix $m \in \R$, $\epsilon>0$ and $\lambda \ge 0$.
There exist constants $K = K(\lambda) >0$ such that
\begin{equation*}
\Phi_{\epsilon}(m; \x^*) \le  \widehat{F}\big(\lambda, |m|, m\big) + \frac{\lambda \epsilon^2}{2} + \frac{K}{N}.
\end{equation*}
\end{proposition}
\begin{proof}
To obtain a bound on $\Phi_{\epsilon}(m;\x^*)$ we use the interpolation method with Hamiltonian
\begin{align*}
-H_{t}(\x) &:=  \sum_{i < j} \sqrt{\frac{t\lambda}{N}} W_{ij}x_ix_j + \frac{t\lambda}{N}x_ix_i^*x_jx_j^* -\frac{t\lambda}{2N} x_i^2x_j^2\\
&~+\sum_{i=1}^N  \sqrt{(1-t)\lambda |m|} z_{i}x_i +  (1-t)\lambda m x_ix_i^* -\frac{(1-t)\lambda |m|}{2} x_i^2 .
\end{align*}
by varying $t\in[0,1]$. The r.v.'s $W, z$ are all i.i.d.\ standard Gaussians independent of everything else. We define
\[\varphi(t) :=  \frac{1}{N} \E_{\W,\z} \log \int  \indi\{\x \in A\} e^{- H_{t}(\x)} \rmd P_{\xtt}^{\otimes N}(\x).\]
We compute the derivative w.r.t.\ $t$, and use Gaussian integration by prts to obtain
\begin{align*}
\varphi'(t) = & -\frac{\lambda}{4} \E\left\langle (R_{1,2}-|m|)^2\right\rangle_{t} + \frac{\lambda t}{4} |m|^2 + \frac{\lambda}{4N^2} \sum_{i=1}^N \E\left\langle {x_i^{(1)}}^2{x_i^{(2)}}^2\right\rangle_{t} \\
&~~~+\frac{\lambda}{2} \E\left\langle (R_{1,*}- m)^2\right\rangle_{t} - \frac{\lambda}{2} m^2 - \frac{\lambda}{2N^2} \sum_{i=1}^N \E\left\langle {x_i}^2{x_i^{*}}^2\right\rangle_{t},
\end{align*}
where $\langle \cdot \rangle_{t}$ is the Gibbs average w.r.t.\ the Hamiltonian $-H_{t}(\x) +\log \indi\{\x \in A\}$.
A few things now happen. Notice that the planted term (first term in the second line) is trivially smaller than $\lambda \epsilon^2/2$ due to the overlap restriction.
Moreover, the last terms in both lines are of order $1/N$ since the variables $x_i$ are bounded. The first term in the first line, which involves the overlap between two replicas, is more challenging.
What makes this term difficult to control is that the Gibbs measure $\langle \cdot \rangle_{t}$ no longer satisfies the Nishimori property due to the overlap restriction, so the overlap between two replicas no longer has the same distribution as the overlap of one replica with the planted spike.
Fortunately, this term is always non-positive so we can ignore it altogether and obtain an upper bound:
\begin{equation*}
\varphi'(t) \le - \frac{\lambda}{4} m^2 +\frac{\lambda \epsilon^2}{2}+ \frac{\lambda  K}{N}.
\end{equation*}
Integrating over $t$, we get
\[\Phi_{\epsilon}(m;\x^*) \le \varphi(0) - \frac{\lambda}{4} m^2 +\frac{\lambda \epsilon^2}{2} + \frac{\lambda K}{N}.\]
Finally, by dropping the indicator, we have
\begin{align*}
\varphi(0) &= \frac{1}{N} \E_{\z} \log \int  \indi\{\x \in A\} e^{- H_{0}(\x)} \rmd P_{\xtt}^{\otimes N}(\x)\\
&\le  \frac{1}{N} \E_{\z} \log \int e^{- H_{0}(\x)} \rmd P_{\xtt}^{\otimes N}(\x)\\
&= \frac{1}{N} \sum_{i=1}^N \E_{z} \log \int \exp\left(\sqrt{\lambda |m|} z x_i +  \lambda m x x_i^* -\frac{\lambda |m|}{2} x^2\right) \rmd P_{\xtt}(x)\\
&=\frac{1}{N} \sum_{i=1}^N \widehat{\psi}(\lambda |m|,\lambda mx_i^*).
\end{align*}
\end{proof}

\subsection{Convergence in probability of the overlaps}
As explained earlier, Propositions~\ref{lower_bound_f} and~\ref{fixed_overlap_upper_bound_paramagnetic} imply convergence in probability of the overlaps:
\begin{proposition}\label{convergence_in_probability_paramagnetic}
For all $\lambda < \lambda_c$ and $\epsilon >0$, there exist constants $K = K(\lambda,\epsilon) \ge 0,~ c = c(\lambda,\epsilon,P_{\xtt}) \ge 0$ such that
\[\E \left\langle \indi\left\{|R_{1,*}| \ge \epsilon \right\}\right\rangle \le K e^{-cN}.\]
\end{proposition}

To prove the above proposition we first show that the partition function of the model enjoys sub-Gaussian concentration on a logarithmic scale.
This is an elementary consequence of two classical concentration-of-measure results: concentration of Lipschitz functions of Gaussian random variables, and concentration of \emph{convex} Lipschitz functions of \emph{bounded} random variables.

\begin{lemma} \label{sub_gaussian_concentration_disorder}
Fix $\x^* \in \R^N$ and let $A$ be a Borel subset of $\R^N$. Define the random variable
\[Z := \int_A e^{-H(\x)}  \rmd P_{\xtt}^{\otimes N}(\x),\]
where the randomness comes from the Gaussian disorder $\W$.
There exists a constant $K>0$ depending on $\lambda$ and $P_{\xtt}$ such that for all $u \ge 0$,
\[\Pr\left(\abs{ \log Z -  \E\log Z } \ge N u\right) \le 2e^{-Nu^2/K}.\]
\end{lemma}

\begin{proof}
We notice that the map $\W \mapsto \frac{1}{N}\log Z$ is Lipschitz with constant $K\sqrt{\frac{\lambda}{N}}$ for every $\x^* \in \R^N$. Then we invoke the Borell-Tsirelson-Ibragimov-Sudakov inequality of concentration of Lipschitz functions of Gaussian r.v.'s. See~\cite{boucheron2013concentration}.
\end{proof}

\begin{lemma} \label{sub_gaussian_concentration_planted_vector}
  Define the random variable
  \[\mathtt{f} := \frac{1}{N}\E_{\W}\log \int e^{-H(\x)}  \rmd P_{\xtt}^{\otimes N}(\x),\]
  where the randomness comes from the planted vector $\x^*$.
  There exist a constant $K>0$ depending on $\lambda$ and $P_{\xtt}$ such that for all $u \ge 0$,
  \[\Pr\left(\abs{ \mathtt{f} -  \E\mathtt{f} } \ge u\right) \le 2e^{-Nu^2/K}.\]
\end{lemma}

\begin{proof}
We notice that the map $\x^* \mapsto \mathtt{f}$ is Lipschitz with constant $K \frac{\lambda}{\sqrt{N}}$ and convex. Moreover, the coordinates $x^*_i$ are bounded. We then invoke Talagrand's inequality on the concentration of convex Lipschitz functions of bounded r.v.'s. See~\cite{boucheron2013concentration}.
\end{proof}

\begin{lemma} \label{sub_gaussian_concentration_sums_of_iid}
There exists a constant $K>0$ depending on $\lambda, m$ and $P_{\xtt}$ such that for all $u \ge 0$,
\[\Pr\left(\abs{ \sum_{i=1}^N \widehat{\psi}(\lambda |m|,\lambda mx_i^*) - \widebar{\psi}(\lambda |m|,\lambda m)} \ge N u\right) \le 2e^{-Nu^2/K}.\]
\end{lemma}

\begin{proof}
Since $\abs{\partial_s\widehat{\psi}(r,sx^*)} \le K^2$, $\abs{\partial_r\widehat{\psi}(r,sx^*)} \le K^2/2$ and $\widehat{\psi}(0,0)=0$, where $K$ is a bound on the radius of the support of $P_{\x}$,
we have $\abs{\widehat{\psi}(r,sx^*)} \le K^2(r/2+s)$. The claim now follows from Hoeffding's inequality.
\end{proof}

\begin{proofof}{Proposition~\ref{convergence_in_probability_paramagnetic}}
For $\epsilon,\epsilon'>0$, we can write the decomposition
\begin{align*}
\E\left\langle \indi\left\{ |R_{1,*}| \ge \epsilon \right\} \right\rangle  &= \sum_{l \ge 0} \E\left\langle \indi \big\{ R_{1,*}-\epsilon \in [l\epsilon',(l+1)\epsilon') \big\}\right\rangle\\
& ~+\sum_{l \ge 0} \E\left\langle \indi \big\{ -R_{1,*}-\epsilon \in [l\epsilon',(l+1)\epsilon') \big\}\right\rangle,
\end{align*}
where the integer index $l$ ranges over a finite set of size $ \le K/\epsilon'$ since the prior $P_{\xtt}$ has bounded support. We will only treat the first sum in the above expression since the argument extends trivially to the second sum. Let $A = \big\{ R_{1,*}-\epsilon \in [l\epsilon',(l+1)\epsilon') \big\}$ and write
\begin{equation}\label{indicator_decomposition_22}
\E\left\langle \indi\{\x \in A\}\right\rangle = \E_{\x^*}\E_{\W} \left[\frac{\int_A e^{-H(\x)}  \rmd P_{\xtt}^{\otimes N}(\x) }{\int e^{-H(\x)}  \rmd P_{\xtt}^{\otimes N}(\x)}\right].
\end{equation}
By virtue of Lemma~\ref{sub_gaussian_concentration_disorder} the two quantities in this fraction enjoy sub-Gaussian concentration on a logarithmic scale over the Gaussian disorder. For any given $l$ and $u \ge 0$, we simultaneously have
\begin{align*}
\frac{1}{N}\log\int e^{-H(\x)}  \rmd P_{\xtt}^{\otimes N}(\x) &\ge \frac{1}{N}\E_{\W} \log\int e^{-H(\x)}  \rmd P_{\xtt}^{\otimes N}(\x) - u 
\end{align*}
and
\begin{align*}
\frac{1}{N}\log \int_A e^{-H_t(\x)}  \rmd P_{\xtt}^{\otimes N}(\x)
&\le \frac{1}{N}\E_{\W}\log \int_A e^{-H_t(\x)}  \rmd P_{\xtt}^{\otimes N}(\x) + u\\
&= \Phi_{\epsilon'}(\epsilon + l\epsilon' ;\x^*)+u,
\end{align*}
with probability at least $1-2e^{-Nu^2/K}$. On the complement of this event, we simply bound the fraction in~\eqref{indicator_decomposition_22} by 1. Combining the above bounds we obtain
\[\E\left\langle \indi\{\x \in A\} \right\rangle \le 2e^{-Nu^2/K} + \E_{\x^*}\left[e^{N(\Delta+2u)}\right],\]
where
\[\Delta = \Phi_{\epsilon'}(m;\x^*) -  \frac{1}{N}\E_{\W} \log\int e^{-H(\x)}  \rmd P_{\xtt}^{\otimes N}(\x),\]
with $m = \epsilon+ l\epsilon'$.
By Proposition~\ref{fixed_overlap_upper_bound_paramagnetic}, $\Phi_{\epsilon'}$ is upper-bounded by a quantity that concentrates over the randomness of $\x^*$.
We use Lemma~\ref{sub_gaussian_concentration_planted_vector} and Lemma~\ref{sub_gaussian_concentration_sums_of_iid} in the same way we used Lemma~\ref{sub_gaussian_concentration_disorder}: for $u'\ge 0$, we simultaneously have
\[\Phi_{\epsilon'}(m;\x^*) \le F(\lambda,|m|,m) +\frac{\lambda \epsilon^2}{2}+ \frac{\lambda K}{N} + u',\]
and
\[\frac{1}{N}\E_{\W} \log\int e^{-H(\x)}  \rmd P_{\xtt}^{\otimes N}(\x) \ge f_N - u',\]
with probability at least $1-4e^{-Nu'^2/K}$, where
\[f_N = \E_{\W,\x^*} \log \int e^{-H(\x)}\rmd P_{\xtt}^{\otimes N}(\x) = \E_{\P_{\lambda}} \log L(\Y;\lambda).\]
Moreover, by Lemma~\ref{asymmetry_lemma}, we have $F(\lambda,|m|,m) \le F(\lambda,|m|,|m|) \equiv F(\lambda,m)$.
 Therefore
\[\E_{\x^*}\left[e^{N\Delta}\right] \le \exp\left(N(F(\lambda,|m|)-f_N+2u')\right)+4e^{-Nu'^2/K}.\]
The second term is obtained by considering the low-probability complement event and noting that $\Delta \le 0$.
Now, by Proposition~\ref{lower_bound_f}, $f_N \ge \sup_{q} F(\lambda,q) - K/N$. 
When $\lambda < \lambda_c$, $q=0$ is the unique maximizer of the $\RS$ potential. Therefore $F(\lambda,|m|)-f_N < -c(\epsilon) <0$ for all $|m| > \epsilon$.
We obtain
\[\E\left\langle \indi\{\x \in A\} \right\rangle \le 2e^{-Nu^2/K} + 4e^{-Nu'^2/K+2Nu} + e^{N(-c(\epsilon)+2u+2u')}.\]
Finally, adjusting the parameters $u,u'$ yields the desired result (e.g., $u' = c(\epsilon)/3$ and $u = c(\epsilon)^2/36 \wedge c(\epsilon)/9$).
\end{proofof}

\vspace{.3cm}
\begin{proofof}{Proposition~\ref{orthogonality}}
Here we prove possibility of strong detection above $\lambda_c$. From Proposition~\ref{lower_bound_f}, we know that $\underline{\lim}\frac{1}{N}\E_{\P_{\lambda}}\log L \ge \phi_{\RS}(\lambda) >0$  for $\lambda > \lambda_c$. On the other hand, $\E_{\P_{0}}\log L \le 0$ by Jensen's inequality. Now it remains to argue that $\frac{1}{N}\log L$ concentrates about its expectation under both $\P_{\lambda}$ and $\P_{0}$. This is a consequence of   
Lemmas~\ref{sub_gaussian_concentration_disorder} and~\ref{sub_gaussian_concentration_planted_vector}: we have, for all $u \ge 0$,
\[\P_{\lambda}\left(\log L - \E_{\P_{\lambda}}\log L \le -Nu\right) \vee \P_{0}\left(\log L - \E_{\P_{0}}\log L \ge Nu\right) \le 4e^{-Nu^2/K}. \]
This concludes the proof. (Note also that the tail decays fact enough to insure almost-sure convergence via the Borel-Cantelli lemma.)  
\end{proofof}

\section{When the diagonal is not discarded}
\label{sxn:diagonal_kept}
When the variance of the diagonal noise entries $W_{ii}$ is kept finite, one has to keep track of the contribution of the diagonal part $d(\x)=\frac{1}{\sigma^2}\sum_{i=1}^N \sqrt{\frac{\lambda}{N}} Y_{ii}x_i^2 -\frac{\lambda}{2N} x_i^4$ of the Hamiltonian. In this case, the derivative of the characteristic function $\phi_N(\lambda)$ of the log-LR w.r.t.\ $\lambda$ displayed in Lemma~\ref{lemma_derivative_f} has an additional term:
\[\phi_N'(\lambda) =  \frac{\imnb s -s^2}{4}  \E\big[(N\langle R_{1,*}^2\rangle - \langle x_N^2x_N^{*2}\rangle)e^{\imnb s \log L}\big] 
+\frac{\imnb s -s^2}{2\sigma^2} \E\big[\langle x_N^2x_N^{*2}\rangle e^{\imnb s \log L}\big].\]
The cavity computations performed in Section~\ref{sxn:asymptotic_decoupling} also need to be altered in a minor way: in the interpolation argument separating the last variable $x_N$ from the rest of the variables, we also have to make $d(\x)$ time-dependent by performing the change of variable $\lambda \to \lambda t$. As a result of the computation, Equation~\eqref{first_fundamental_bound} becomes
\[(1-\lambda)N \E\left[\langle R_{1,*}^2\rangle e^{\imnb s \log L}\right] = \E\left[\langle x_N^2x_N^{*2}\rangle e^{\imnb s \log L}\right] +\frac{\lambda \kappa}{\sigma^2} \E\left[e^{\imnb s \log L}\right]+ \delta,\]   
with $|\delta|\le K/\sqrt{N}$, $\kappa = \E_{P_{\xtt}}[X^3]^2$, while Equation~\eqref{second_fundamental_bound} remains intact. As a result of these changes, and the above formula for $\phi_N'$, we get
\[\phi_N'(\lambda) = \frac{\imnb s - s^2}{4} \left(\frac{1+\lambda \kappa/\sigma^2}{1-\lambda} - 1\right)\phi_N(\lambda)  + \frac{\imnb s - s^2}{2\sigma^2}\phi_N(\lambda)+ \delta,\] 
and this leads to the formula claimed.
\section{Conclusion}
This paper investigates the fundamental limits of spike detection in the rank-one spiked Wigner model. We proved that the logarithm of the likelihood ratio has Gaussian fluctuations below the reconstruction threshold $\lambda_c$ while it is exponentially large under the alternative above it. This establishes the maximal region of contiguity between the planted and null models: namely the open interval $(0,\lambda_c)$. This also pins down the performance of the optimal test, and provides formulae for the Kullback-Leibler and the total variation distances between the null and planted distributions. An important characteristic of this threshold is that it is not necessarily related to the spectrum of the observed matrix: there are cases where $\lambda_c$ does not correspond to the point where the signal shows up in the spectrum.       

Our proofs repose on the technology developed within spin-glass theory for the study of the SK model. It is of interest to extend these techniques to  other spiked models, notably spiked covariance models where the perturbation is in the covariance matrix of the data. Partial progress establishing Gaussian fluctuations of the LR  in a restricted regime was recently obtained by a subset of the authors~\citep{alaoui2018detection}. Reaching the optimal threshold---a conjectural formula of which is provided in this recent paper---remains an interesting problem.

\appendix

\section{Convergence of the second moment}
\label{sxn:convergence_second_moment}
In this section we prove the convergence of the second moment of the overlaps: $\E\langle R_{1,*}^2 \rangle \le \frac{K}{N}$. 
We recall the notation $\nu_t(f) = \E\langle f\rangle_{t}$, where $\langle \cdot \rangle_t$ denotes the interpolating Gibbs average corresponding to the Hamiltonian 
\begin{align*}
-H_t(\x) &:= \sum_{1\le i < j \le N-1}  \sqrt{\frac{\lambda }{N}} W_{ij}x_ix_j + \frac{\lambda }{N}x_ix_i^*x_jx_j^* -\frac{\lambda }{2N} x_i^2x_j^2\\
&~~+\sum_{i =1}^{N-1} \sqrt{\frac{\lambda t}{N}} W_{iN}x_i x_N + \frac{\lambda t}{N} x_ix_i^* x_Nx_N^* - \frac{\lambda t}{2N} x_i^2 x_N^2.
\end{align*}
The following lemma will be useful.
\begin{lemma}\label{first_second_order_estimates}
For all $t \in [0,1]$, and all $\tau_1,\tau_2 >0$ such that $1/\tau_1+1/\tau_2=1$,
\begin{align}
\abs{\nu_t(f) - \nu_0(f)} &\le K(\lambda,n) \nu\left(\abs{R^-_{1,*}}^{\tau_1}\right)^{1/\tau_1} \cdot \nu\left(|f|^{\tau_2}\right)^{1/\tau_2} \label{first_order_estimate}\\
\abs{\nu_t(f) -\nu_0(f) - \nu_0'(f)} &\le K(\lambda,n) \nu\left(\abs{R^-_{1,*}}^{\tau_1}\right)^{1/\tau_1} \cdot \nu\left(|f|^{\tau_2}\right)^{1/\tau_2}.\label{second_order_estimate}
\end{align}
\end{lemma}

\begin{proof}
We use Taylor's approximations
\begin{align*}
\abs{\nu_t(f) - \nu_0(f)} &\le \sup_{0 \le s \le 1} \abs{\nu'_s(f)},\\
\abs{\nu_s(f) -\nu_0(f) - \nu_0'(f)} &\le \sup_{0 \le s \le 1} \abs{\nu''_s(f)},
\end{align*}
then Lemma~\ref{gibbs_derivative} and the triangle inequality to bound the right hand sides, then H\"older's inequality to bound each term in the derivative, and then we apply Lemma~\ref{bound_gibbs_derivative}. (To compute the second derivative, one need to use Lemma~\ref{gibbs_derivative} recursively.)
\end{proof}

By symmetry between variables, we have
\[\E\langle R_{1,*}^2 \rangle = \E\langle x_N x_N^* R_{1,*}\rangle = \frac{1}{N}\E\langle (x_N x_N^*)^2\rangle + \E\langle x_N x_N^* R^-_{1,*}\rangle.\]
By the first bound~\eqref{first_order_estimate} of Lemma~\ref{first_second_order_estimates} with $\tau_1 = 1$, $\tau_2 = \infty$, we have
\[\nu((x_Nx_N^*)^2) =  \nu_0((x_Nx_N^*)^2) + \delta = \E_{P_{\xtt}}[X^2]^2 + \delta = 1+\delta,\]
with $|\delta| \le K(\lambda)\nu(|R^-_{1,*}|)$.
On the other hand, by the second bound~\eqref{second_order_estimate} with $\tau_1 = 1$, $\tau_2 = \infty$, we get
\[\nu (R^-_{1,*}x_Nx_N^*) = \nu_0'(R^-_{1,*}x_Nx_N^*) + \delta. \]
This is because $\nu_0(R^-_{1,*}x_Nx_N^*)=0$, since last variable $x_N$ decouples from the remaining $N-1$ variables under the measure $\nu_0$. Now, we use Lemma~\ref{gibbs_derivative} with $n=1$, to evaluate the above derivative at $t=0$. We still write $y^{(l)} = x_N^{(l)}$.
\begin{align*}
\nu_0'(R^-_{1,*}x_Nx_N^*) &= -\lambda \nu_0(y^{(1)}y^{(2)}y^{(1)}y^*R^-_{1,*}R^-_{1,2})
+ \lambda \nu_0(y^{(1)}y^{*}y^{(1)}y^{*}R^{-2}_{1,*})\\
&~~~ - \lambda  \nu_0(y^{(2)}y^{*}y^{(1)}y^* R^-_{1,*}R^-_{2,*})
+ \lambda  \nu_0(y^{(2)}y^{(3)}y^{(1)}y^*R^-_{1,*}R^-_{2,3})\\
&= \lambda \nu_0(R^{-2}_{1,*}).
\end{align*}
In the above, the only term that survived is the second one since all variables $y$ appearing in it are squared.
We now use Lemma~\ref{first_second_order_estimates} to argue that $\nu_0(R^{-2}_{1,*}) \simeq \nu_1(R^{-2}_{1,*})$.
We apply the estimate~\eqref{first_order_estimate} with $t=1$, $\tau_1 = 3$ and $\tau_2 = 3/2$ to obtain
\[\nu_0(R^{-2}_{1,*}) = \nu(R^{-2}_{1,*}) + \delta\]
with $|\delta| \le K(\lambda)\nu(|R^-_{1,*}|^3)$. Moreover,
\[\nu(R^{-2}_{1,*}) = \nu((R_{1,*}-\frac{1}{N}yy^*)^2) = \nu(R^{2}_{1,*}) -\frac{2}{N}\nu(yy^*R_{1,*}) + \frac{1}{N^2}\nu((yy^*)^2).\]
The third term is of order $1/N^2$, and the second term is bounded by $\frac{1}{N}\nu(|R_{1,*}|)$. Therefore
\[\nu_0((R^-_{1,*})^2) = \nu(R_{1,*}^2) + \delta',\]
with \[|\delta'| \le K(\lambda)\left(\frac{1}{N}\nu(|R^-_{1,*}|)+\nu(|R^-_{1,*}|^3) + \frac{1}{N^2}\right).\]
Putting things together, we have proved that
\begin{equation}\label{first_cavity}
\nu(R_{1,*}^2) = \frac{1}{N} + \lambda \nu(R_{1,*}^2) + \delta,
\end{equation}
where
\begin{equation}\label{delta_error}
|\delta| \le  K(\lambda)\left(\frac{1}{N}\nu(|R^-_{1,*}|)+\nu(|R^-_{1,*}|^3) + \frac{1}{N^2}\right).
\end{equation}
Now we need to control the error term $\delta$. By elementary manipulations,
\[\nu(|R_{1,*}^-|) \le \nu(|R_{1,*}|) + \frac{K}{N},\]
and
\[\nu(|R_{1,*}^-|^3) \le \nu(|R_{1,*}|^3) + \frac{K}{N}\nu(R_{1,*}^2) + \frac{K}{N^2}\nu(|R_{1,*}|) + \frac{K}{N^3}.\]
Therefore, from~\eqref{delta_error} we have
\begin{equation}\label{delta_error_2}
|\delta| \le K \left( \nu(|R_{1,*}|^3) +  \frac{1}{N}\nu(R_{1,*}^2) + \frac{1}{N}\nu(|R_{1,*}|) +\frac{1}{N^2} \right).
\end{equation}
At this point, the prior knowledge that $R_{1,*}$ is small with high probability is useful. It implies that
 $\nu(|R_{1,*}|) \ll 1$ and $\nu(|R_{1,*}|^3) \ll \nu(R_{1,*}^2)$.
With Proposition~\ref{convergence_in_probability_paramagnetic} we have for $\epsilon>0$
\[\nu(|R_{1,*}|)  \le \epsilon + K(\epsilon) e^{-c(\epsilon)N},\]
and
\[\nu( |R_{1,*}|^3)  \le \epsilon \nu( R_{1,*}^2) + K(\epsilon) e^{-c(\epsilon)N}.\]
Combining the above two bounds with~\eqref{delta_error_2}, and then injecting in~\eqref{first_cavity}, we get
\begin{align*}
\nu(R_{1,*}^2) \le \frac{1}{N} + (\lambda + \frac{K}{N} +K\epsilon) \nu(R_{1,*}^2) + \frac{K}{N^2}+Ke^{-cN}.
\end{align*}
 We choose $\epsilon$ small enough and $N$ large enough that $K(\epsilon+\frac{1}{N}) <1 -\lambda$. We obtain
\[\nu\left(R_{1,*}^2\right) \le \frac{K}{N} + \frac{K}{N^2} +Ke^{-cN}.\]

\section{Convergence of the fourth moment}
\label{sxn:convergence_fourth_moment}
In this section we prove the convergence of the fourth moment:
$\E\langle R_{1,*}^4\rangle \le \frac{K}{N^2}$.
We adopt the same technique based on the cavity method, with the extra knowledge that the second moment converges.
Many parts of the argument are exactly the same so we will only highlight the main novelties in the proof.
By symmetry between variables,
\begin{align*}
\nu(R_{1,*}^4) = \nu \left(R_{1,*}^3x_Nx_N^*\right)
&=  \nu((R^-_{1,*})^3 x_Nx_N^*) + \frac{3}{N}\nu ((R^-_{1,*})^2 (x_Nx_N^*)^2)\\
&~~~ + \frac{3}{N^2}\nu (R^-_{1,*}(x_Nx_N^*)^3) + \frac{1}{N^3}\nu((x_Nx_N^*)^4).
\end{align*}
The quadratic term is bounded as
\[\nu((R^-_{1,*})^2 (x_Nx_N^*)^2) \le K \nu((R^-_{1,*})^2) \le \frac{K}{N}. \]
The last inequality is using our extra knowledge about the convergence of the second moment. The last two terms are also bounded by $K/N^2$ and $K/N^3$ respectively.
Now we must deal with the cubic term, and here, we apply the exact same technique used to deal with the term $\nu (R^-_{1,*}x_Nx_N^*)$ in the previous proof.
The argument applies verbatim  and we obtain
\begin{align*}
\nu(R_{1,*}^4) \le \frac{K}{N^2}+\lambda \nu(R_{1,*}^4) + K\nu(|R^-_{1,*}|^{5}) + K\nu(|R^-_{1,*}|^{3})
\end{align*}
Using Proposition~\ref{convergence_in_probability_paramagnetic}, we have for $\epsilon >0$,
\[\nu( |R_{1,*}|^5)  \le \epsilon \nu( R_{1,*}^4) + K(\epsilon) e^{-c(\epsilon) N},\]
\[\nu( |R_{1,*}|^3)  \le \epsilon \nu( R_{1,*}^2) + K(\epsilon) e^{-c(\epsilon) N}.\]
Now, we finish the argument in the same way, by choosing $\epsilon$ sufficiently small.
This concludes the proof of Theorem~\ref{convergence_fourth_moment_symmetric_paramagnetic}.


\section{Proof of Lemma~\ref{asymmetry_lemma}}
\label{sxn:proof_of_lemma_six}
A straightforward calculation reveals that 
\begin{equation*} 
\frac{\partial}{\partial s} \widebar{\psi}(r,s) = \E\left[\langle xx^*\rangle\right], 
\quad \mbox{and} \quad 
\frac{\partial^2}{\partial s^2} \widebar{\psi}(r,s) = \E\left[x^{* 2}(\langle x^2\rangle-\langle x\rangle^2)\right] >0,
\end{equation*} 
so that $s \mapsto \frac{\partial}{\partial s} \widebar{\psi}(r,s)$ is Lipschitz and strongly convex on any interval, and for all $r\ge 0$. �

Let $\nu = P_{\xtt}$, and let $\mu$ be the symmetric part of $P_{\xtt}$, i.e., $\mu(A) = (P_{\xtt}(A)+P_{\xtt}(-A))/2$ for all Borel $A \subseteq \R$. We observe that $\nu$ is absolutely continuous with respect to $\mu$, so that the Radon-Nikodym derivative $\frac{\rmd \nu}{\rmd \mu}$ is a well-defined measurable function from $\R$ to $\Rp$ that integrates to one.

\begin{lemma}\label{asymmetry_lower_bound2}
For all $r \ge 0$, we have 
\[\widebar{\psi}(r,r)  -\widebar{\psi}(r,-r) \ge 2\E \left[\left\langle \frac{\rmd \nu}{\rmd \mu}(x) - 1 \right\rangle_{\mu,r}^2\right],\]
where $\langle \cdot \rangle_{\mu,r}$ is the average w.r.t.\ to the Gibbs measure corresponding to the Gaussian channel $y = \sqrt{r} x^* + z$, $x^* \sim \mu$ and $z \sim \normal(0,1)$.
\end{lemma}

\begin{proofof}{Lemma~\ref{asymmetry_lower_bound2}}
The argument relies on a smooth interpolation method between the two measures $\mu$ and $\nu$. Let $t \in [0,1]$ and let $\rho_t = (1-t)\mu + t\nu$. Further, let $r,s \ge 0$ be fixed, and
\[\widebar{\psi}(r,s;t) := \E_{z} \int \left(\log \int \exp\left(\sqrt{r}zx + s xx^* - \frac{r}{2} x^2\right) \rmd \rho_t(x)\right) \rmd \rho_t(x^*),\]
where $z \sim \normal(0,1)$. Now let 
\[\phi(t) = \widebar{\psi}(r,r;t) - \widebar{\psi}(r,-r;t).\]
We have $\phi(1) = \widebar{\psi}(r,r) - \widebar{\psi}(r,-r)$ on the one hand, and since $\mu$ is a symmetric distribution, $\phi(0) = 0$ on the other. 
We will show that $\phi$ is a convex increasing function on the interval $[0,1]$, strictly so if $\mu \neq \nu$, and that $\phi'(0) = 0$. Then we  deduce that $\phi(1) \ge \frac{\phi''(0)}{2}$, allowing us to conclude. First, we have
\begin{align*}
\frac{\rmd}{\rmd t} \widebar{\psi}(r,r;t) &= \E_{z} \int \log \int e^{\sqrt{r}zx + r xx^* - \frac{r}{2} x^2} \rmd \rho_t(x) ~\rmd (\nu - \mu)(x^*)\\
&~~~+\E_{z} \int \frac{\int e^{\sqrt{r}zx + r xx^* - \frac{r}{2} x^2}\rmd(\nu - \mu)(x)}{\int e^{\sqrt{r}zx + r xx^* - \frac{r}{2} x^2}\rmd \rho_t(x)} ~\rmd \rho_t(x^*),
\end{align*}
and
\begin{align*}
\frac{\rmd^2}{\rmd t^2} \widebar{\psi}(r,r;t) &= 2\E_{z} \int \frac{\int e^{\sqrt{r}zx + r xx^* - \frac{r}{2} x^2}\rmd(\nu - \mu)(x)}{\int e^{\sqrt{r}zx + r xx^* - \frac{r}{2} x^2}\rmd \rho_t(x)} ~\rmd (\nu - \mu)(x^*)\\
&~~~ -2\E_{z} \int \left(\frac{\int e^{\sqrt{r}zx + r xx^* - \frac{r}{2} x^2}\rmd(\nu - \mu)(x)}{\int e^{\sqrt{r}zx + r xx^* - \frac{r}{2} x^2}\rmd \rho_t(x)}\right)^2 ~\rmd \rho_t(x^*).
\end{align*}
Similar expressions holds for $\widebar{\psi}(r,-r;t)$ where $x^*$ is replaced by $-x^*$ inside the exponentials. We see from the expression of the first derivative at $t=0$ that $\widebar{\psi}(r,r;0)' = \widebar{\psi}(r,-r;0)'$. This is because $\rho_0 = \mu$ is symmetric about the origin, so a sign change (of $x$ for the first term, and $x^*$ for the second term in the expression) does not affect the value of the integrals. Hence $\phi'(0) = 0$. 
Now, we focus on the second derivative. Observe that since $\mu$ is the symmetric part of $\nu$, $\nu - \mu$ is anti-symmetric. This implies that the first term in the expression of the second derivative changes sign under a sign change in $x^*$, and keeps the same modulus. As for the second term, a sign change in $x^*$ induces integration against $\rmd \rho_t(-x^*)$. Hence we can write the difference $(\widebar{\psi}(r,r;t)-\widebar{\psi}(r,-r;t))''$ as  
\begin{align*}
\frac{\rmd^2}{\rmd t^2} \phi(t) &= 4 \E_{z} \int \frac{\int e^{\sqrt{r}zx + r xx^* - \frac{r}{2} x^2}\rmd(\nu - \mu)(x)}{\int e^{\sqrt{r}zx + r xx^* - \frac{r}{2} x^2}\rmd \rho_t(x)} ~\rmd (\nu - \mu)(x^*)\\ 
&~~~-2\E_{z} \int \left(\frac{\int e^{\sqrt{r}zx + r xx^* - \frac{r}{2} x^2}\rmd(\nu - \mu)(x)}{\int e^{\sqrt{r}zx + r xx^* - \frac{r}{2} x^2}\rmd \rho_t(x)}\right)^2 ~(\rmd \rho_t(x^*)-\rmd \rho_t(-x^*)).
\end{align*}
For any Borel $A$, we have $\rho_t(A)- \rho_t(-A) = (1-t)(\mu(A) - \mu(-A)) + t(\nu(A)-\nu(-A)) = 2t(\nu - \mu)(A)$. Therefore the second term in the above expression becomes
\[-4t\E_{z} \int \left(\frac{\int e^{\sqrt{r}zx + r xx^* - \frac{r}{2} x^2}\rmd(\nu - \mu)(x)}{\int e^{\sqrt{r}zx + r xx^* - \frac{r}{2} x^2}\rmd \rho_t(x)}\right)^2 ~\rmd(\nu - \mu)(x^*).\]
Since both $\mu$ and $\nu$ are absolutely continuous with respect to $\rho_t$ for all $0 \le t < 1$ we write
\[\frac{\rmd^2}{\rmd t^2} \phi(t) = 4 \E_{z,x^*} \left\langle \frac{\rmd(\nu - \mu)}{\rmd\rho_t}(x) \frac{\rmd(\nu - \mu)}{\rmd\rho_t}(x^*) \right\rangle - 4 t\E_{z,x^*} \left\langle \frac{\rmd(\nu - \mu)}{\rmd\rho_t}(x)\right\rangle^2,\]
where the Gibbs average is with respect to the posterior of $x$ given $z,x^*$ under the Gaussian channel $y = \sqrt{r}x^* + z$, and the expectation is under $x^* \sim \rho_t$ and $z \sim \normal(0,1)$. By the Nishimori property, we simplify the above expression to
\[\frac{\rmd^2}{\rmd t^2} \phi(t) = 4(1-t)\E \left[\left\langle \frac{\rmd(\nu - \mu)}{\rmd\rho_t}(x)\right\rangle^2\right],\]
where the expression is valid for all $0\le t <1$. From here we see that the function $\phi$ is convex on $[0,1]$ (where we have closed the right end of the interval by continuity). Since $\phi(0) = \phi'(0) = 0$, $\phi$ is also increasing on $[0,1]$. Therefore we have
\[\phi(1) \ge \half\phi''(0) = 2\E \left[\left\langle \frac{\rmd \nu}{\rmd\mu}(x)-1\right\rangle_{\mu,r}^2\right].\]    
\end{proofof}

\vspace{5mm}
\textbf{Acknowledgments.}
We are grateful to L\'{e}o Miolane for insightful conversations and to Nike Sun for comments on an earlier version of this manuscript. We warmly thank the anonymous reviewers of their feedback.
This research was initiated at the \emph{Workshop on Statistical physics, Learning, Inference and Networks} at Ecole de Physique des Houches, winter 2017.
This work was supported by the Multidisciplinary University Research Initiative under Army Research Office grant number W911NF-17-1-0304.

\vspace{5mm}
\begin{small}
\bibliographystyle{apalike}
\bibliography{phase_transitions}
\end{small}

\end{document}